\documentclass[11pt]{amsart}
\usepackage{amsmath}
\usepackage{amssymb}
\usepackage{amsfonts}
\usepackage{amsthm}
\usepackage{enumerate}
\usepackage[mathscr]{eucal}
\usepackage{mathrsfs}
\usepackage{graphicx}

\usepackage[bookmarksnumbered, bookmarksopen, colorlinks, citecolor=blue, linkcolor=blue]{hyperref}

\newtheorem{theorem}{Theorem}[section]
\newtheorem{lemma}[theorem]{Lemma}

\theoremstyle{definition}
\newtheorem{definition}[theorem]{Definition}

\newtheorem{proposition}[theorem]{Proposition}

\theoremstyle{remark}

\numberwithin{equation}{section}

\def\intslash{\rlap{\kern  .32em $\mspace {.5mu}\backslash$ }\int}
\def\qsl{{\rlap{\kern  .32em $\mspace {.5mu}\backslash$ }\int_{Q_x}}}

\def\R{{\mathbb R}}
\def\Z{{\mathbb Z}}

\def\C{{\mathcal C}}

\def\K{\mathcal{K}}

\def\diam{{\text{\it  diam}}}

\def\pari{\partial}

\def\ga{\gamma}

\def\dist{{\text{\it dist}}}

\def\supp{{\text{\rm supp}}}

\def\inn#1#2{\langle#1,#2\rangle}

\def\rta{\rightarrow}

\def\card{\text{\rm card}}
\def\lc{\lesssim}

\def\pv{\text{\rm p.v.}}
\def\alp{\alpha}

\def\del{\delta}             
\def\eps{\varepsilon}

\def\tet{\theta}

\def\lam{\lambda}             \def\Lam{\Lambda}
\def\si{\sigma}              

\def\om{\omega}              \def\Om{\Omega}
\def\fr{\frac}

\newcommand{\Be}{\begin{equation}}
\newcommand{\Ee}{\end{equation}}
\newcommand{\Bes}{\begin{equation*}}
\newcommand{\Ees}{\end{equation*}}
\newcommand{\Bsp}{\begin{split}}
\newcommand{\Esp}{\end{split}}
\newcommand{\Bm}{\begin{multline}}
\newcommand{\Em}{\end{multline}}
\newcommand{\Bea}{\begin{eqnarray}}
\newcommand{\Eea}{\end{eqnarray}}
\newcommand{\Beas}{\begin{eqnarray*}}
\newcommand{\Eeas}{\end{eqnarray*}}
\newcommand{\Benu}{\begin{enumerate}}
\newcommand{\Eenu}{\end{enumerate}}
\newcommand{\Bi}{\begin{itemize}}
\newcommand{\Ei}{\end{itemize}}

\begin{document}
\title[Weak $(1,1)$ estimate for maximal truncated operator]{Weak $(1,1)$ estimate for maximal truncated rough singular integral operator}

\author[]{Xudong Lai}
\address{Xudong Lai:
Institute for Advanced Study in Mathematics\\
Harbin Institute of Technology\\
Harbin
150001\\
China;
Zhengzhou Research Institute\\
Harbin Institute of Technology\\
Zhengzhou
450000\\
China}

\email{xudonglai@hit.edu.cn}
\thanks{This work is supported by National Natural Science Foundation of China (No. 12322107, No. 12271124 and No. W2441002) and Heilongjiang Provincial Natural Science Foundation of China (No. YQ2022A005).}
\subjclass[2010]{Primary  42B20; Secondary 42B25}
\keywords{Maximal operator, singular integral operator, rough kernel, weak $(1,1)$}

\begin{abstract}
In their seminal work (Amer. J. Math.
  78: 289-309, 1956),  Calder\'on and Zygmund introduced  the maximal truncated rough singular integral operator and established its $L^p$-boundedness for $1 < p < \infty$. However, the endpoint case $p = 1$ remained an open problem. This paper resolves this problem. More precisely, we prove that the maximal truncated rough singular integral operator is of  weak type $(1,1)$.
\end{abstract}

\maketitle

\section{Introduction}
\vskip0.24cm
Let $\Omega$ be a homogeneous function of degree zero on $\mathbb{R}^d \setminus \{0\}$, meaning $\Omega(\theta) = \Omega(r\theta)$ for all $\theta \in S^{d-1}$ and $r > 0$. Assume that $\Omega$ is integrable on the sphere ($\Omega \in L^1(S^{d-1})$) and satisfies the cancellation condition
\begin{align}\label{e:27canca}
 \int_{S^{d-1}} \Omega(\theta) d\sigma(\theta) = 0,
\end{align}
where $d\sigma(\theta)$ denotes the surface measure on $S^{d-1}$.

In their seminal work \cite{CZ56}, Calder\'on and Zygmund introduced the maximal truncated singular integral operator with rough kernel $\Omega$:
\begin{align*}
   T_{\Omega,*} f(x) = \sup_{\epsilon > 0} \Big| \int_{|x-y| > \epsilon} \frac{\Omega(x-y)}{|x-y|^d} f(y) dy \Big|.
\end{align*}
Using the method of rotations, they established the $L^p$-boundedness of $T_{\Omega,*}$ for $1 < p < \infty$ when either $\Omega$ is odd and $\Omega \in L^1(S^{d-1})$, or $\Omega$ is even and $\Omega \in L\log L(S^{d-1})$.

However, the weak type $(1,1)$ boundedness of $T_{\Om,*}$ has remained an open problem since then, even for $\Omega \in L^\infty(S^{d-1})$. This longstanding problem has been highlighted recently  several times, see for example by Seeger \cite{See14} in 2014.

In this paper, we resolve this open problem. Our main result can be stated  as  follows.
\begin{theorem}\label{t:27main}
  Suppose $d\geq2$, $\Omega$ is a homogeneous function of degree zero on $\mathbb{R}^d\backslash \{0\}$ satisfying \eqref{e:27canca} and $\Omega \in L\log L(S^{d-1})$.  Then the operator $T_{\Omega,\ast}$ is of weak type $(1, 1)$, i.e. for each $f \in L^1(\mathbb{R}^d)$ and each $\lambda >0$, the following estimate holds
\begin{align*}
  \lambda |\{x\in \mathbb{R}^d: T_{\Omega,\ast} f(x)>\lambda\}| \lesssim \C_\Omega \|f\|_{L^1(\mathbb{R}^d)},
\end{align*}
where $\C_\Omega$ is a finite constant (see its definition in \eqref{e:27ccon}).
\end{theorem}

The challenges in resolving this open problem arise from two fundamental aspects: the roughness of the kernel and the maximal nature of $T_{\Omega,*}$. Even for linear rough operator, establishing its weak type $(1,1)$ boundedness remains highly non-trivial.
The study of boundedness for linear rough operator originated with Calder\'on and Zygmund \cite{CZ56} in 1956, who introduced the singular integral operator with homogeneous kernel $\Omega$
\begin{align*}
T_\Om f(x) = \pv\int \frac{\Omega(x-y)}{|x-y|^d}f(y)dy
\end{align*}
and established its $L^p$-boundedness for $1<p<\infty$ under the assumption either $\Omega$ is odd and $\Om\in L^1(S^{d-1})$ or $\Omega$ is even and $\Om\in L\log L(S^{d-1})$.
It was until in 1988 that Christ and Rubio de Francia \cite{CR88} obtained its weak type $(1,1)$ boundedness if $\Om\in L\log L(S^{1})$ in the two-dimensional case (independent by Hofmann \cite{Hof88} with $\Om\in L^q(S^{1})$ for $1<q\leq\infty$). Both these two works were motivated by Christ's previous $TT^*$ method in \cite{Chr88}. Later in 1996 Seeger \cite{See96} utilized the microlocal decomposition approach to establish its weak type $(1,1)$ boundedness for all dimensions if $\Om\in L\log L(S^{d-1})$.  In 1999,  Tao \cite{Tao1999} extended the $TT^*$ method  and obtained the weak type $(1,1)$ boundedness for rough singular integral operator on homogenous group.

Our main effort in this paper is devoted to linearizing the maximal operator. To this end, we make a standard  Calder\'on-Zygmund decomposition, then it suffices to deal with estimates for bad functions. To obtain some necessary decay estimates related to these bad functions,  we must linearize the maximal operator. Various techniques exist for linearizing maximal operators.
Our strategy is to use the Rademacher-Menshov theorem (see Lemma \ref{l:l4}). Preliminarily we will make physical dyadic decompositions for both the kernel $\frac{\Omega(x)}{|x|^d}$ and the function $f$ as follows
$$T_J f(x)=\int_{\mathbb{R}^d} \mathcal{K}_j (x-y)(f\chi_{\fr12J})(y)dy$$
so that $\supp(T_Jf)\subseteq J$,  here $J$ is a dyadic cube.
Hence for a fixed $x$, the supremum in $T_{*}f(x)$ (see \eqref{e:27MT} for its definition) should be taken as a sum over  dyadic cubes containing the point $x$, which form a natural net  around this point.
By employing an iteration technique, we will construct a suitable partition of the dyadic cubes across the whole space. This partition allows each decomposition of the dyadic cubes to be reorganized into a natural net (as analyzed above) over which the maximal operator can be linearized by the Rademacher-Menshov theorem.
After linearization, we conduct a meticulous analysis of the relationship between the dyadic cubes from the physical space and those originating from the Calder\'on-Zygmund decomposition.  This analysis, together with some decay estimates for the linearized operator,  will  yield the required bounds for bad functions.

\subsection*{Outline of the paper} In Section \ref{s:272}, by carrying out a reduction, the dyadic decomposition and the Calder\'on-Zygmund decomposition, the proof of our main theorem is reduced to establishing a key decay estimate for bad functions (see Proposition \ref{l:l3}).
In Section~\ref{s:273} we mainly present the linearization of the maximal operator and derive the required decay estimate for these bad functions based on some estimates for the linearized operator (see claim \eqref{e:e11}).
Finally the proof of estimates for the linearized operator is rather lengthy.
We will decompose the linearized operator into four operators and reduce the overall problem to establishing four key lemmas (see Lemma \ref{l:27finall21}, Lemma \ref{l:27finall22}, Lemma \ref{l:27finall23} and Lemma \ref{l:27finall24}). Section \ref{s:274} is devoted to the proofs of Lemma \ref{l:27finall21}, Lemma \ref{l:27finall22} and Lemma \ref{l:27finall23}. We will  prove Lemma \ref{l:27finall24} separately in Section \ref{s:275}.

\subsection*{Notation} Throughout this paper, we only consider the dimension $d\geq 2$ case and  $C$ stands for a positive finite constant which is independent of the essential variables, not necessarily the same one in each occurrence. $A\lc B$ means $A\leq CB$ for some constant $C$. By the notation $C_\eps$  we mean that the constant depends on the parameter $\eps$. $A\approx B$ means that $A\lc B$ and $B\lc A$.
$\Z$ denotes the set of all integers and $\Z^d=\Z\times \cdots\times \Z$ with $d$-tuple products. $\Z_+$ stands  for the set of all nonnegative integers.
For any measurable set $E\subseteq\R^d$, we denote by $|E|$ the Lebesgue measure of $E$.
$\chi_E$ represents the characteristic  function of $E$.
For any cube $Q$ and $s>0$, let $l(Q)$ be the sidelength of $Q$ and $sQ$ be the cube with the same center as $Q$ and $l(sQ)=sl(Q)$.
Define $\langle f \rangle_{Q}=\frac{1}{|Q|}\int_Q |f(x)| dx.$
For any dyadic cube $I$, $J$ and $K$, we set their sidelengths as $2^i$, $2^j$ and $2^k$ respectively. Define $B(x,r)$ as a ball with center $x$ and radius $r$.
For any $1\leq r\leq\infty$, set $r'$ as the dual number, i.e. $\fr{1}{r}+\fr{1}{r'}=1$. Denote $dist(E,F)=\inf\{|x-y|:x\in E,y\in F\}$ and $diam(E)=\sup\{|x-y|:x,y\in E\}$ for $E,F\subseteq \R^d$. We define $\|\Om\|_{\infty}\triangleq\|\Om\|_{L^\infty(S^{d-1})}$, $\|\Om\|_q\triangleq\big(\int_{S^{d-1}}|\Om(\tet)|^qd\sigma(\tet)\big)^{\fr{1}{q}}$ and $\|\Om\|_{L\log^+L(S^{d-1})}\triangleq\int_{S^{d-1}}|\Om(\tet)|\log(2+|\Om(\tet)|)d\sigma(\tet).$
Denote by $\mathcal{F}f$ (or $\hat{f}$) and $\mathcal{F}^{-1}f$ (or $\check{f}$)  the Fourier transform and the inversion Fourier transform
of $f$  which are given by
$$\mathcal{F}f(\xi)=\int_{\R^d} e^{-i\inn{x}{\xi}}f(x)dx,\ \ \ \ \mathcal{F}^{-1}f(\xi)=\fr{1}{(2\pi)^{d}}\int_{\R^d}e^{i\inn{x}{\xi}}{f(x)dx}.$$

\vskip0.24cm
\section{Proof of Theorem \ref{t:27main}: some basic estimates}\label{s:272}
\vskip0.24cm

In this section, we give some  basic and standard estimates for  weak $(1,1)$ boundedness of the maximal truncated operator $T_{\Omega,*}$. We will apply the Calder\'on-Zygmund decomposition and reduce our proof to estimates for bad functions.

\vskip0.24cm

\subsection{A reduction}
We first reduce the study of $T_{\Omega,*}$ to a maximal dyadic truncated operator.
Let $\varphi$ be a smooth function on $\mathbb{R}^d$ which is supported in the annulus $\{2^{-4} < |x| < 2^{-2}\}$ and satisfies the partition of unity condition
\begin{equation*}
\sum_{j\in \mathbb{Z}}\varphi_j(x) = 1 \quad \text{for all } x \in \mathbb{R}^d \setminus \{0\},
\end{equation*}
where $\varphi_j(x) = \varphi(2^{-j}x)$. We define the associated dyadic operator
\begin{align*}
T_j f(x) = \int_{\mathbb{R}^d} \mathcal{K}_j(x-y)f(y) dy
\end{align*}
with the kernel $\mathcal{K}_j(x) = \varphi_j(x)\frac{\Omega(x)}{|x|^d}$.

A straightforward estimation yields the pointwise control: for $x \in \mathbb{R}^d$,
\begin{align}\label{e:27point1}
T_{\Omega,*} f(x) \lesssim M_{\Omega}f(x) + T_*f(x)
\end{align}
and
\begin{align}\label{e:27point2}
T_*f(x)\lesssim M_{\Omega}f(x) + T_{\Omega,*} f(x)
\end{align}
where $M_\Omega$ and $T_*$ are defined by
\Be\label{e:27MT}
M_\Omega f(x)=\sup_{r>0}\fr{1}{r^d}\int_{B(x,r)}|\Omega(x-y)f(y)|dy; \quad T_*f(x)=\sup_{l\in \mathbb{Z}} \Big|\sum_{j\geq l}T_j f(x)\Big|.
\Ee

Notice that $T_*$ is of strong type $(p,p)$ for $1<p<\infty$ which is a consequence of the pointwise estimate \eqref{e:27point2}
together with the fact $M_\Omega$ is  $L^p$ bounded if $\Omega\in L^1(S^{d-1})$ and $T_{\Omega,*}$ is  $L^p$ bounded if $\Omega\in L\log L(S^{d-1})$ (see e.g. \cite{CZ56} or \cite{Gra14C}). On the other hand,
it is  known that $M_\Omega$ is of weak type $(1,1)$ if $\Omega\in L\log L(S^{d-1})$ (see \cite{CR88}).
Hence by the pointwise estimate \eqref{e:27point1}, to prove $T_{\Omega,*}$ is of weak type $(1,1)$, it suffices to show that  $T_*$ is of weak type $(1,1)$ which we  restate  as follows.

\begin{theorem}\label{t:27maind}
Suppose $\Omega$ is a homogeneous function of degree zero on $\mathbb{R}^d\backslash \{0\}$ satisfying \eqref{e:27canca} and $\Omega \in L\log L(S^{d-1})$. Then for any $f\in L^1(\mathbb{R}^d)$ and $\lambda>0$, the following estimate holds
\begin{align*}%\label{e:27weak}
  \lambda |\{x\in \mathbb{R}^d: T_{\ast} f(x)>\lambda\}| \lesssim \C_\Omega \|f\|_{L^1(\mathbb{R}^d)}
\end{align*}
where $\C_\Omega$ is a finite constant.
\end{theorem}

\vskip0.24cm

\subsection{Further dyadic decomposition} Before proceeding further, we introduce some dyadic systems. Let $\mathfrak{D}$ be a set of standard dyadic cubes in $\mathbb{R}^d$, i.e.
$$\mathfrak{D}=\Big\{\prod_{j=1}^d[m_j2^k,(m_j+1)2^k):\ {(m_1,m_2,\cdots,m_d)}\in\Z^d,\ k\in\Z\Big\}.$$

For each $\overrightarrow{w}\in \{0,\fr12\}^d$, define $\mathfrak{D}^{\overrightarrow{w}}$ as the standard dyadic grid shifted by $\overrightarrow{w}$, i.e. $\mathfrak{D}^{\overrightarrow{w}}$ is the set of these dyadic cubes
$$\Big\{2^k\overrightarrow{w}+\prod_{j=1}^d[m_j2^k,(m_j+1)2^k):\ {(m_1,m_2,\cdots,m_d)}\in\Z^d,\ k\in\Z\Big\}.$$
Then it is easy to get the identity
\begin{align*}
  \sum_{\overrightarrow{w}\in\{0,\fr12\}^d}\sum_{K\in \mathfrak{D}^{\overrightarrow{w}}:l(K)=2^k} \chi_{\fr12 K} = 1,\ \ \text{for every}\ k\in \mathbb{Z}.
\end{align*}

Applying this identity, we see that
\Be\label{e:27id}
\begin{split}
  T_k g(x)
  & =\sum_{\overrightarrow{w}\in\{0,\fr12\}^d}\sum_{K\in \mathfrak{D}^{\overrightarrow{w}}:l(K)=2^k} \int_{\mathbb{R}^d} \mathcal{K}_k (x-y)(g\chi_{\fr12K})(y)dy\\
  & \triangleq \sum_{\overrightarrow{w}\in\{0,\fr12\}^d}\sum_{K\in \mathfrak{D}^{\overrightarrow{w}}:l(K)=2^k} T_K g(x),
\end{split}
\Ee
where for each $K\in \mathfrak{D}^{\overrightarrow{w}}$ with $l(K)=2^k$, $T_K g(x)$ is defined by
$$\int_{\mathbb{R}^d} \mathcal{K}_k (x-y)(g\chi_{\fr12K})(y)dy.$$
Notice that by the support of $\mathcal{K}_k$ and $g\chi_{\fr12K}$, we get that $T_Kg(x)$ is supported in $K$:
\Be\label{e:27supp}
\supp(T_Kg)\subseteq K.
\Ee
This is an important property of $T_Kg(x)$ which will be frequently used in our later proof.

We also point out that the following  three helpful properties of  dyadic cubes in $\mathfrak{D}$:

(D-1)\ For any $K\in\mathfrak{D}$, $l(K)=2^{k}$ where $k\in\Z$;

(D-2)\ $K\cap J\in\{K,J,\emptyset\}$ for any $K,J\in\mathfrak{D}$;

(D-3)\ these cubes of a fixed sidelength of $2^k$ form a partition of $\R^d$.

The dyadic grid $\mathfrak{D}^{\overrightarrow{w}}$ satisfies the above properties (D-1) and (D-3). However, for $\vec{w} \neq (0,0)$, the intersected cubes in this grid are not necessarily nested; that is, the property (D-2) may fail.  Nevertheless if  bisecting each sides of a cube $K\in\mathfrak{D}^{\overrightarrow{w}}$,  we then get $2^d$
congruent dyadic cubes and this new family of dyadic cubes satisfy the property (D-2).
This strategy in this manner is very useful in our later proof.
\vskip0.24cm

\subsection{Calder\'on-Zygmund decomposition}
Let us consider the function $f$ and the constant $\lambda$ given in Theorem \ref{t:27maind} and fix $f,\lambda$ in the rest of this paper.
Define
\Be\label{e:27ccon}
\mathcal{C}_\Om\triangleq\|\Om\|_{L\log L(S^{d-1})}+\int_{S^{d-1}}|\Om(\tet)|\big(1+\log^+({|\Om(\tet)|}/{\|\Om\|_{1}})\big)d\sigma(\tet).
\Ee
Since $\|\Om\|_{L\log^+L(S^{d-1})}<+\infty$, one can easily check that $\mathcal{C}_\Om$ is a finite constant.

By performing the Calder\'on-Zygmund decomposition of $f$ at level $\lambda/{\C_\Omega}$ in the dyadic system $\mathfrak{D}$ (see \cite{Gra14C}), we get a countable set of  dyadic cubes $\mathcal{Q}\subseteq\mathfrak{D}$ and the following conclusions:

\begin{align*}
  \tag{cz-i} & f=h+b,\ \|h\|_\infty \lesssim {\lam}/{\C_\Omega},\ \ \|h\|_{L^1(\mathbb{R}^d)}\lc\|f\|_{L^1(\mathbb{R}^d)};  \\
  \tag{cz-ii} & b=\sum_{Q\in\mathcal{Q}}b_Q,\ \text{each $b_Q$ satisfies }\int_Q b_Q(x)dx = 0,\ \supp (b_Q) \subseteq Q;  \\
  \tag{cz-iii} & \forall Q \in \mathcal{Q}: \|b_Q\|_{L^1(\mathbb{R}^d)} \lesssim \langle f \rangle_Q |Q|, \  \langle f\rangle_{Q}\approx {\lam}/{\C_\Omega}; \\
  \tag{cz-iv} & \text{All dyadic cubes in $\mathcal{Q}$ are disjoint}; \\
  \tag{cz-v} & \text{Let $E = \bigcup_{Q \in \mathcal{Q}} Q$. Then $|E| \lesssim \frac{\C_{\Omega}}{\lambda}\|f\|_{L^1(\mathbb{R}^d)}$}. \\
\end{align*}

Now we start to prove Theorem \ref{t:27maind}. Using the property (cz-i), we decompose $f=h+b$ and obtain
$$|\{x\in\mathbb{R}^d: T_{*}f(x)>\lambda\}|\leq |\{x\in\mathbb{R}^d:T_{*}h(x)>\lambda/2\}|+|\{x\in\mathbb{R}^d:T_{*}b(x)>\lambda/2\}|.$$
By the Chebyshev inequality, the fact $T_*$ is $L^2$-bounded with operator norm at most $C\|\Om\|_{L\log^+L(S^{d-1})}$ and the property (cz-i), we get
\Bes
\begin{split}
|\{x\in\mathbb{R}^d:T_{*}h(x)>\lambda/2\}|&\lc {\lambda ^{-2}}\|T_{*}h\|_{L^2(\mathbb{R}^d)}^2\\
&\lc{\lambda^{-2}}(\|\Om\|_{L\log^+L(S^{d-1})}\|h\|_{L^2(\mathbb{R}^d)})^2\lc{\lambda}^{-1}\C_\Om\|f\|_{L^1(\mathbb{R}^d)}.
\end{split}
\Ees

Set $E^*=\bigcup_{Q\in \mathcal{Q}}2^{300}Q$. Then we see
\[|\{x\in\mathbb{R}^d:T_{*}b(x)>\lambda/2\}|\leq |E^*|+|\{x\in (E^*)^c:T_{*}b(x)>\lambda/2\}|.\]
By properties (cz-iv) and (cz-v), the set $E^*$ satisfies
$
|E^*|\lc |E|\lc{\lambda}^{-1}\C_\Om\|f\|_{L^1(\mathbb{R}^d)}.
$
Thus, to complete the proof of Theorem \ref{t:27maind}, it remains to show
\begin{equation}\label{e:weak}
|\{x\in (E^*)^c:T_{*}b(x)>\lambda/2\}|\lc {\lambda}^{-1}\C_\Om\|f\|_{L^1(\mathbb{R}^d)}.
\end{equation}

\vskip0.24cm

\subsection{Estimates related to bad functions} Set $\mathcal{Q}_s=\{Q\in \mathcal{Q}:l(Q)=2^s\}$. Define
$b_s(x) = \sum_{Q \in \mathcal{Q}_s} b_Q(x)$. Then $b(x) = \sum_{s\in\Z} b_s(x)$.
We write
$$T_*b(x)=\sup_{l\in \mathbb{Z}} \Big|\sum_{k\geq l}\sum_{s\in\Z}T_k b_{k-s}(x)\Big|.$$

Note that $T_kb_{k-s}(x)=0$ if $x\in (E^*)^c$ and $s<200$. Therefore we obtain
\begin{equation}\label{e:27Tjbj}
\begin{split}
\big|\big\{x\in (E^*)^c:&\,T_{*}b(x)>{\lambda}/{2}\big\}\big|\\
&=\Big|\Big\{x\in (E^*)^c:\sup_{l\in \mathbb{Z}} \big|\sum_{k\geq l}\sum_{s\geq200}T_kb_{k-s}(x)\big|>{\lambda}/{2}\Big\}\Big|.\\
\end{split}
\end{equation}
Using the equality \eqref{e:27id} and the triangle inequality, we get
\Bes
\sup_{l\in \mathbb{Z}} \Big|\sum_{k\geq l}\sum_{s\geq200}T_kb_{k-s}(x)\Big|\leq \sum_{\overrightarrow{w}\in\{0,\fr12\}^d}\sum_{s\geq200}\sup_{l\in \mathbb{Z}}\Big|\sum_{K\in\mathfrak{D}^{\overrightarrow{w}}:l(K)\geq 2^l}T_Kb_{k-s}(x)\Big|,
\Ees
here and in the sequel, $k=k(K)$ is the integer such that $l(K)=2^k$. Substituting the above estimate into \eqref{e:27Tjbj}, we obtain
\begin{equation*}
\begin{split}
\big|\big\{x\in (E^*)^c:&\,T_{*}b(x)>{\lambda}/{2}\big\}\big|\\
&\leq\sum_{\overrightarrow{w}\in\{0,\fr12\}^d}\Big|\Big\{x\in \mathbb{R}^d:\sum_{s\geq200}\sup_{l\in \mathbb{Z}}\big|\sum_{K\in\mathfrak{D}^{\overrightarrow{w}}:l(K)\geq 2^l}T_Kb_{k-s}(x)\big|>2^{-d-1}{\lambda}\Big\}\Big|.\\
\end{split}
\end{equation*}

Hence to prove \eqref{e:weak}, it is enough to establish the estimate below
\Be\label{e:27weakf}
\Big|\Big\{x\in \mathbb{R}^d:\sum_{s\geq200}\sup_{l\in \mathbb{Z}}\big|\sum_{K\in\mathfrak{D}^{\overrightarrow{w}}:l(K)\geq 2^l}T_Kb_{k-s}(x)\big|>2^{-d-1}{\lambda}\Big\}\Big|\lc {\lambda}^{-1}\C_\Om\|f\|_{L^1(\mathbb{R}^d)}
\Ee
for each $\overrightarrow{w}\in\{0,\fr12\}^d$. We fix $\overrightarrow{w}\in\{0,\fr12\}^d$ in our later proof.

Let $s\geq200$ be the integer in \eqref{e:27weakf}. In the following, we make a decomposition of the homogeneous function $\Omega$
$$\Omega(\tet)=\Omega(\tet)\chi_{\{|\Om(\tet)|> 2^{\eta s}\|\Om\|_1\}}+\Omega(\tet)\chi_{\{|\Om(\tet)|\leq 2^{\eta s}\|\Om\|_1\}}\triangleq\Omega_1(\tet)+\Omega_2(\tet) $$
where $\eta$ is a positive sufficiently small constant to be chosen later. Hence we split the kernel $\mathcal{K}_k$ into two parts
$$\mathcal{K}_k(x) = \varphi_k(x)\frac{\Omega_1(x)}{|x|^d}+\varphi_k(x)\frac{\Omega_2(x)}{|x|^d}\triangleq \mathcal{K}_{k,1}(x)+\mathcal{K}_{k,2}(x).$$
Consequently we decompose the operator $T_K$ into two parts
$$T_K g(x)=\mathcal{K}_{k,1}*(g\chi_{\fr12K})(x)+\mathcal{K}_{k,2}*(g\chi_{\fr12K})(x)\triangleq T_{K,1} g(x)+T_{K,2} g(x).$$

For $T_{K,1}$, we have the following simple $L^1$ estimate.
\begin{proposition}\label{l:27L1}
With all notions above, we get
$$\sum_{s\geq200}\sum_{K\in\mathfrak{D}^{\overrightarrow{w}}}\|T_{K,1}b_{k-s}\|_{L^1(\mathbb{R}^d)}\lc\C_\Om\|f\|_{L^1(\mathbb{R}^d)}.$$
\end{proposition}
\begin{proof}
A straightforward estimation yields
$$\|T_{K,1}b_{k-s}\|_{L^1(\mathbb{R}^d)}\lc\int_{S^{d-1}}|\Omega_1(\tet)|d\sigma(\tet)\|b_{k-s}\chi_{\fr12K}\|_{L^1(\mathbb{R}^d)}.$$
Therefore by properties (cz-iii), (cz-iv) and (cz-v) in the Calder\'on-Zygmund decomposition, we derive
\begin{equation*}
\begin{split}
\sum_{s\geq200}&\sum_{K\in\mathfrak{D}^{\overrightarrow{w}}}\|T_{K,1}b_{k-s}\|_{L^1(\mathbb{R}^d)}\lc
\sum_{s\geq200}\sum_{k\in\Z}\|b_{k-s}\|_{L^1(\mathbb{R}^d)}\int_{S^{d-1}}|\Omega_1(\tet)|d\sigma(\tet)\\&
\lc\|f\|_{L^1(\mathbb{R}^d)}\int_{S^{d-1}}\card\big\{s\in \Z_+:\, s\geq 200, 2^{\eta s}< |\Om(\tet)|/\|\Om\|_1\big\}|\Om(\tet)|d\sigma(\tet)\\
&\lc\|f\|_{L^1(\mathbb{R}^d)}\int_{S^{d-1}}|\Om(\tet)|\big(1+\log^+(|\Om(\tet)|/\|\Om\|_1)\big)d\sigma(\tet)\lc\C_\Om\|f\|_{L^1(\mathbb{R}^d)}
\end{split}
\end{equation*}
which completes the proof.
\end{proof}
For the term $T_{K,2}$, the following $L^2$ estimate with an exponential decay in $s$ constitutes the crucial part of our proof.

\begin{proposition}\label{l:l3}
  With all notions above,  for any $s\geq200$,  there exists a constant $\delta >0$ such that
\begin{align*}
  \Big\|\sup_{l \in\Z}\big|\sum_{K \in \mathfrak{D}^{\overrightarrow{w}},l(K)\geq 2^l } T_{K,2} b_{k-s}\big| \Big\|_{L^2(\mathbb{R}^d)} \leq \big(s^22^{-\delta s}\C_\Omega\lambda\|f\|_{L^1(\mathbb{R}^d)})^{\fr12}.
\end{align*}
\end{proposition}
The proof of Proposition \ref{l:l3} will be presented in the next section. Applying Proposition \ref{l:27L1} and Proposition \ref{l:l3}, we can finish the proof of \eqref{e:27weakf} as follows. Splitting $T_K$ as $T_{K,1}$ and $T_{K,2}$, together with the Chebyshev inequality and the triangle inequality,
\Bes
\Big|\Big\{x\in \mathbb{R}^d:\sum_{s\geq200}\sup_{l\in \mathbb{Z}}\big|\sum_{K\in\mathfrak{D}^{\overrightarrow{w}}:l(K)\geq 2^l}T_Kb_{k-s}(x)\big|>2^{-d-1}{\lambda}\Big\}\Big|\lc I+II
\Ees
where
\Bes
\begin{split}
I&={\lambda}^{-1}\sum_{s\geq200}\sum_{K\in\mathfrak{D}^{\overrightarrow{w}}}
\|T_{K,1}b_{k-s}\|_{L^1(\mathbb{R}^d)}\\
\end{split}
\Ees
and
\Bes
\begin{split}
II&={\lambda}^{-2}  \Big(\sum_{s\geq200}\Big\|\sup_{l \in\Z}\big|\sum_{K \in \mathfrak{D}^{\overrightarrow{w}},l(K)\geq 2^l } T_{K,2} b_{k-s}\big| \Big\|_{L^2(\mathbb{R}^d)}\Big)^2.
\end{split}
\Ees

By Proposition \ref{l:27L1}, $I$ is bounded by $\C_\Omega\lam^{-1}\|f\|_{L^1(\mathbb{R}^d)}$. By Proposition \ref{l:l3}, $II$ is also bounded by $\C_\Omega\lam^{-1}\|f\|_{L^1(\mathbb{R}^d)}$. So we prove \eqref{e:27weakf}. Hence we complete the proof of Theorem \ref{t:27maind} based on Proposition \ref{l:l3}.

\vskip 0.24cm
\section{Proof of Proposition \ref{l:l3}}\label{s:273}

In this section, we give the proof of Proposition \ref{l:l3}. Let $s\geq200$ be a fixed integer in the rest of the proof.
With slight abuse of notation, we still use $T_K$, $\mathcal{K}_k$ and $\Omega$ to represent $T_{K,2}$, $\mathcal{K}_{k,2}$ and $\Omega_2$ respectively which will not cause confusions in this and later section. From the definition of $\Omega_2$, we only need to assume that  $\|\Omega\|_\infty\leq2^{\eta s}\|\Omega\|_1$.

We begin by presenting some preliminary lemmas in Subsection~\ref{s:27s31}. Subsection~\ref{s:2733} gives a partition of the dyadic cubes appearing in the sum for the maximal function in Proposition \ref{l:l3} and states a key result (see Proposition \ref{l:27exclude}). Using Proposition \ref{l:27exclude}, we then prove Proposition \ref{l:l3} in Subsection~\ref{s:2734}. Finally, the proof  for the linearization of the maximal operator in Proposition \ref{l:27exclude} is provided in Subsection~\ref{s:2735}. While  the proofs for estimates of the linearized operator will be given in Section \ref{s:274} and Section \ref{s:275}.

\subsection{Some preliminary lemmas}\label{s:27s31}

The first lemma  is the H\"ormander-Mihlin multiplier theorem with explicit bounds which can be found in \cite{Gra14C}.
\begin{lemma}\label{l:27Hmihlin}
Let $m$ be a complex-value bounded function on $\R^d\setminus\{0\}$ which satisfies
$$|\pari_\xi^\alp m(\xi)|\leq \mathcal{A}|\xi|^{-|\alp|}$$
for all multi indices $|\alp|\leq [\fr{d}{2}]+1$. Then the operator $T_m$ defined by
$$\widehat{T_mg}(\xi)=m(\xi)\hat{g}(\xi)$$
is of strong type $(p,p)$ for $1<p<\infty$ with bound $C_d(\mathcal{A}+\|m\|_{L^\infty(\R^d)})$.
\end{lemma}

One important technique in our later proof is to linearize the maximal operator.
To that end,  we need the following Rademacher-Menshov theorem (see e.g. \cite[Theorem 10.6]{DTT08}).
\begin{lemma}\label{l:l4}
  Let $(X, \mu)$ be a measure space and $\{f_j\}_{j=1}^N$ be a sequence of measurable functions satisfying the  Bessel-type inequality:
  For any finite sequence $\{\epsilon_j\}_{j=1}^N$ with each $\epsilon_j\in \{-1,1\}$,
  \begin{align*}
    \Big\|\sum_{j=1}^{N}\epsilon_j f_j\Big\|_{L^2(X)} \leq B.
  \end{align*}
  Then the following maximal inequality holds
  \begin{align*}
    \Big\|\sup_{0<M\leq N}\big|\sum_{j=1}^{M} f_j\big| \Big\|_{L^2(X)} \lesssim B\log(2+N).
  \end{align*}
\end{lemma}
\vskip0.24cm

\subsection{A partition of dyadic cubes}\label{s:2733}
Now we come back to consider the proof of Proposition \ref{l:l3}. We first ignore the maximal function in Proposition \ref{l:l3} and make an appropriate decomposition of cubes $K \in \mathfrak{D}^{\overrightarrow{w}}$ in the following sum
\Be\label{e:27mainnon}
\sum_{K \in \mathfrak{D}^{\overrightarrow{w}}} T_{K} b_{k-s}(x).
\Ee
Since $T_{K}b_{k-s}(x)=\mathcal{K}_k*(b_{k-s}\chi_{\fr12K})(x)$, we only need to consider these cubes $K \in \mathfrak{D}^{\overrightarrow{w}}$ satisfying $b_{k-s}\chi_{\fr12K}\neq0$. The following lemma shows that the total measure of $K$ in the above sum is controllable.
\begin{lemma}
For any measurable set $A\subseteq\mathbb{R}^d$, we have the Carleson measure type estimate
  \begin{equation}\label{e:e8}
    \sum_{K\in\mathfrak{D}^{\overrightarrow{w}};K \subseteq A \atop b_{k-s}\chi_{\frac{1}{2}K} \neq 0} |K| \lesssim 2^{ds} |A|.
  \end{equation}
Meanwhile the following uniform estimate holds
\begin{equation}\label{e:einf}
    \sum_{K\in\mathfrak{D}^{\overrightarrow{w}}; b_{k-s}\chi_{\frac{1}{2}K} \neq 0} |K| \lesssim 2^{ds} \lam^{-1}\C_\Omega\|f\|_{L^1(\mathbb{R}^d)}.
  \end{equation}
\end{lemma}
\begin{proof}
Observe that  for each $K\in\mathfrak{D}^{\overrightarrow{w}}$ with $b_{k-s}\chi_{\frac{1}{2}K} \neq 0$,  there exists at least  one cube $Q \in \mathcal{Q}_{k-s}$  such that $b_Q\chi_{\frac{1}{2}K} \neq 0$. Therefore  $Q$ must intersect $\fr{1}{2}K$, here and in the sequel, $Q$ intersects $\fr{1}{2}K$ means their interiors  have a non-empty intersection.
Since $s\geq200$, $K\in\mathfrak{D}^{\overrightarrow{w}}$ with $l(K)=2^k$ and $Q\in\mathfrak{D}$ with $l(Q)=2^{k-s}$,  we get $Q \subsetneq K$ (in fact $Q\subseteq \fr12 K$).
We say such a cube $Q$ is associated with $K$ and denote it as $Q_K$.

We point out all $Q_K$s are disjoint. Indeed, consider any two different dyadic cubes $K,J\in\mathfrak{D}^{\overrightarrow{w}}$ with $b_{k-s}\chi_{\frac{1}{2}K} \neq 0$ and $b_{j-s}\chi_{\frac{1}{2}J} \neq 0$. Then either $l(K)=l(J)$  or $l(K)\neq l(J)$. If $l(K)=l(J)$, then $K\cap J=\emptyset$ hence $Q_K\cap Q_J =\emptyset$ since  the previous observation shows that $Q_K\subseteq K$ and $Q_J\subseteq J$. If the sidelengths of $K$ and $J$ are different, since $Q_K\in\mathcal{Q}_{k-s}$, $Q_J\in\mathcal{Q}_{j-s}$ and all cubes in $\mathcal{Q}$ are disjoint (see (cz-iv)), we get $Q_K\cap Q_J=\emptyset$. So we prove that all $Q_K$s are disjoint.

Let $Q$ be a cube associated with $K$. Then the property (cz-iii) in the Calder\'on-Zygmund decomposition yields the following estimate
  \begin{align*}
    \int_Q |f(y)|dy = |Q|\langle f \rangle_Q \gtrsim |Q|\lam\C_\Omega^{-1} = 2^{-ds} |K| \lam\C_\Omega^{-1}.
  \end{align*}
Since all $Q_K$s associated with $K\in\mathfrak{D}^{\overrightarrow{w}}$  are disjoint, we have
 \Be\label{e:e9}
  \begin{split}
    \sum_{K\in\mathfrak{D}^{\overrightarrow{w}};K \subseteq A \atop b_{k-s}\chi_{\frac{1}{2}K} \neq 0} |K|
    & \lesssim 2^{ds} \lam^{-1}\C_\Omega \sum_{k\in\Z}\sum_{Q\in\mathcal{Q}_{k-s}}\sum_{K\in\mathfrak{D}^{\overrightarrow{w}};l(K)=2^k \atop Q\subsetneq K \subseteq A} \int_Q |f(y)|dy\\
    & \lesssim 2^{ds} \lam^{-1}\C_\Omega \sum_{k\in\Z}\sum_{Q\in\mathcal{Q}_{k-s}; Q\subseteq A} \int_Q |f(y)|dy\\
    &\lc 2^{ds} \lam^{-1}\C_\Omega \sum_{Q\in\mathcal{Q};Q\subseteq A} \lam\C_\Omega^{-1}|Q|\\
    &\lesssim 2^{ds} \min\{|A|,\lam^{-1}\C_\Omega\|f\|_{L^1(\mathbb{R}^d)}\}
  \end{split}
 \Ee
 where  the second inequality follows from that there exists at most one dyadic cube $K\in\mathfrak{D}^{\overrightarrow{w}}$ with fixed sidelength $l(K)=2^k$ such that $K\supsetneq Q$,  while in the  third and fourth  inequalities we use  properties (cz-iii), (cz-iv) and (cz-v) in the Calder\'on-Zygmund decomposition. Hence we  prove \eqref{e:e8} and \eqref{e:einf}.
\end{proof}

Notice the intersected cubes in $\mathfrak{D}^{\overrightarrow{w}}$ may not have the property that one contains the other. To overcome this defect, by bisecting each side of the cube $K\in\mathfrak{D}^{\overrightarrow{w}}$,  we then get $2^d$ disjoint dyadic cubes which are redefined as follows.

\begin{definition}[$K^{\iota}$]
By selecting a fixed sequence order according to their spatial positions, we relabel these $2^d$ dyadic cubes as $K^1, K^2, \ldots, K^{2^d}$.
\end{definition}

For a fixed $1\leq \iota\leq 2^d$, it is easy to see that any two given dyadic cubes in $\{K^{\iota}: K\in\mathfrak{D}^{\overrightarrow{w}}\}$  satisfy  the property: either they are disjoint or one contains the other.  Moreover, if $K^\iota\subseteq J^\iota$ for $K,J\in\mathfrak{D}^{\overrightarrow{w}}$, then $K\subseteq J$.
This follows from the geometric observation that $K^\iota$ and $J^\iota$ occupy the same relative spatial positions within their father cubes $K$ and $J$, respectively.

In what follows, we introduce auxiliary sets $F_{s,\iota}^n$ for integers $n \geq 1$ and $1\leq \iota\leq 2^d$.

\begin{definition}[$F_{s,\iota}^n$]\label{d:27Fsn}
 Define $F_{s,\iota}^1$ as
  \begin{align*}
    F_{s,\iota}^1 \triangleq \Big\{x\in \mathbb{R}^d:\sum_{K\in\mathfrak{D}^{\overrightarrow{w}};  b_{k-s}\chi_{\frac{1}{2}K} \neq 0} \chi_{K^{\iota}}(x) > C_02^{ds}\Big\}
  \end{align*}
where $C_0$ is a large constant to be chosen later. For $n\geq2$, we define the set $F_{s,\iota}^n$ successively as follows:
  \begin{align*}
    F_{s,\iota}^n \triangleq \Big\{x\in \mathbb{R}^d:\sum_{K\in\mathfrak{D}^{\overrightarrow{w}}; K^\iota\subseteq F_{s,\iota}^{n-1} \atop  b_{k-s}\chi_{\frac{1}{2}K} \neq 0} \chi_{K^{\iota}}(x) > C_02^{ds}\Big\}.
  \end{align*}
\end{definition}

We first show how the set $F_{s,\iota}^n$ looks like.
For convenience set $F_{s,\iota}^0=\mathbb{R}^d$.
Since any two dyadic cubes in $\{K^\iota:K\in\mathfrak{D}^{\overrightarrow{w}}\}$  satisfy  the property that either they are disjoint or one contains the other, then  by the definition of $F_{s,\iota}^n (n\geq1)$,
we could observe that for any $x\in F_{s,\iota}^n$, there exists a dyadic cube $K\in\mathfrak{D}^{\overrightarrow{w}}$ such that $K^{\iota}\subseteq F_{s,\iota}^{n-1}$, $b_{k-s}\chi_{\frac{1}{2}K}\neq0$, $x\in K^{\iota}$, $K^{\iota}$ has more than $C_02^{ds}$ ancient dyadic cubes contained in $F_{s,\iota}^{n-1}$ and hence we get $K^\iota\subseteq F_{s,\iota}^n$.
Moreover, for any $x\in F_{s,\iota}^n$, there exists a maximal dyadic cube $K^{\iota}\subseteq F_{s,\iota}^n$ such that all its dyadic subcube $J^\iota$ which satisfies $J^\iota\subseteq K^\iota, J\in\mathfrak{D}^{\overrightarrow{w}}$, $b_{j-s}\chi_{\fr12J}\neq0$ and $J^\iota\subseteq F_{s,\iota}^{n-1}$ must be a subset of $F_{s,\iota}^n$.
Therefore, by choosing the maximal dyadic cubes $K^\iota$ in $F^n_{s,\iota}$, we can write $F_{s,\iota}^n=\bigcup_{K^\iota\in \mathcal{Q}_{s,\iota,n}}{K^\iota}$
where $\mathcal{Q}_{s,\iota,n}$ is a collection of disjoint dyadic cubes.

It is also easy to see that $F_{s,\iota}^1\supseteq F_{s,\iota}^2\supseteq\cdots\supseteq F_{s,\iota}^n\supseteq\cdots$. Regarding their measures, we have the following more refined estimate.

\begin{lemma}\label{l:27Fsn}
For $n\geq1$ and $1\leq \iota\leq 2^d$, the measure of $F_{s,\iota}^n$ satisfies
\Be\label{e:e10}
|F_{s,\iota}^n|\lc 2^{-2n}\lam^{-1}\C_\Om\|f\|_{L^1(\mathbb{R}^d)}.
\Ee
\end{lemma}
\begin{proof}
By the Chebyshev inequality and \eqref{e:einf}, we get for $n=1$
 \begin{align*}
    |F_{s,\iota}^1| \leq \frac{1}{C_02^{ds}}\int\sum_{K \in \mathfrak{D}^{\overrightarrow{w}};\atop b_{k-s}\chi_{\frac{1}{2}K}\neq 0 }\chi_{K^{\iota}}(x)dx = \frac{1}{C_02^{ds+d}} \sum_{K\in\mathfrak{D}^{\overrightarrow{w}};\atop b_{k-s}\chi_{\frac{1}{2}K}\neq 0} |K| \lc\lam^{-1}\C_\Om\|f\|_{L^1(\mathbb{R}^d)}.
  \end{align*}
Similarly for $n\geq2$,
  \begin{align}\label{e:27fnsiota}
    |F_{s,\iota}^n| \leq \frac{1}{C_02^{ds}}\int\sum_{K \in \mathfrak{D}^{\overrightarrow{w}};K^\iota\subseteq F_{s,\iota}^{n-1}\atop
    b_{k-s}\chi_{\frac{1}{2}K}\neq 0}\chi_{K^{\iota}}(x)dx = \frac{1}{C_02^{ds+d}} \sum_{K\in\mathfrak{D}^{\overrightarrow{w}}; K^\iota\subseteq F_{s,\iota}^{n-1} \atop b_{k-s}\chi_{\frac{1}{2}K}\neq 0} |K|.
  \end{align}
Since $F_{s,\iota}^{n-1}$ can be rewritten as a union of the maximal cubes: $F_{s,\iota}^{n-1}=\bigcup_{J^\iota\in\mathcal{Q}_{s,\iota,n-1}}J^\iota$, then for any $K^\iota\subseteq F_{s,\iota}^{n-1}$, there exists a unique maximal dyadic cube $J^\iota\in\mathcal{Q}_{s,\iota,n-1}$ such that $K^\iota\subseteq J^\iota$.
Since the spatial positions of $J^\iota$ and $K^\iota$ are fixed  relative to their father cubes $J$ and $K$, respectively,  it follows that $K\subseteq J$. Hence if we set $\tilde{F}_{s,\iota}^{n-1}=\bigcup_{J^\iota\in\mathcal{Q}_{s,\iota,n-1}}J$, then $K\subseteq\tilde{F}_{s,\iota}^{n-1}$. Therefore \eqref{e:27fnsiota} is majorized by
\begin{align*}
    \frac{1}{C_02^{ds+d}} \sum_{K\in\mathfrak{D}^{\overrightarrow{w}}; K\subseteq \tilde{F}_{s,\iota}^{n-1} \atop b_{k-s}\chi_{\frac{1}{2}K}\neq 0} |K| &\leq 2^{-2-d}|\tilde{F}_{s,\iota}^{n-1}|\leq 2^{-2-d}\sum_{J^\iota\in\mathcal{Q}_{s,\iota,n-1}}|J|\\
    &=2^{-2}\sum_{J^\iota\in\mathcal{Q}_{s,\iota,n-1}}|J^\iota|= 2^{-2}|{F}_{s,\iota}^{n-1}|
  \end{align*}
  where  in the first inequality we use (\ref{e:e8}) and choose the constant $C_0$ large enough, while the last equality follows from $F_{s,\iota}^{n-1}$ can be rewritten as a union of the disjoint cubes $F_{s,\iota}^{n-1}=\bigcup_{J^\iota\in\mathcal{Q}_{s,\iota,n-1}}J^\iota$. Notice that $C_0$ here is independent of $n$, we get $|{F}_{s,\iota}^{n}|\leq 2^{-2}|{F}_{s,\iota}^{n-1}|$ for $n\geq2$. Iterating this estimate, we get \eqref{e:e10}.
\end{proof}
In the following we define a partition of cubes $K\in\mathfrak{D}^{\overrightarrow{w}}$ in the sum \eqref{e:27mainnon}.
\begin{definition}[$\mathcal{I}_{s,\iota}^{\#,n}$]
Let $F_{s,\iota}^n$ be given in Definition \ref{d:27Fsn}. We define $\mathcal{I}_{s,\iota}^{\#,1}$ as the collection of cubes $K \in \mathfrak{D}^{\overrightarrow{w}}$ appearing in the sum \eqref{e:27mainnon} such that $K^\iota$ is not contained in $F_{s,\iota}^1$, that is,
\[
\mathcal{I}_{s,\iota}^{\#,1} \triangleq \big\{K \in \mathfrak{D}^{\overrightarrow{w}} : b_{k-s}\chi_{\frac{1}{2}K} \neq 0,\ K^{\iota} \nsubseteq F_{s,\iota}^1\big\},
\]
where $K^{\iota} \nsubseteq F_{s,\iota}^1$ means either the interior of $K^{\iota}$ is contained in $(F_{s,\iota}^1)^c$ or  intersects both $(F_{s,\iota}^1)^c$ and $F_{s,\iota}^1$.
For $n \geq 2$, we define
\[
\mathcal{I}_{s,\iota}^{\#,n} \triangleq \big\{K \in \mathfrak{D}^{\overrightarrow{w}} : b_{k-s}\chi_{\frac{1}{2}K} \neq 0,\ K^{\iota} \nsubseteq F_{s,\iota}^n,\ K^{\iota} \subseteq F_{s,\iota}^{n-1}\big\},
\]
which consists of cubes $K \in \mathfrak{D}^{\overrightarrow{w}}$ in the sum \eqref{e:27mainnon} such that $K^{\iota}$ is contained in $F_{s,\iota}^{n-1}$ but not in $F_{s,\iota}^n$.
\end{definition}

We now illustrate that the sets $\{\mathcal{I}_{s,\iota}^{\#,n} : n = 1, 2, \dots\}$ form a partition of the cubes $K \in \mathfrak{D}^{\overrightarrow{w}}$ appearing in the sum \eqref{e:27mainnon}. In fact, we need only consider those cubes $K \in \mathfrak{D}^{\overrightarrow{w}}$ for which $b_{k-s} \chi_{\frac{1}{2}K} \neq 0$.

We construct the partition inductively as follows:

    \emph{Step 1}. If $K^{\iota} \nsubseteq F_{s,\iota}^1$, we assign $K$ to $\mathcal{I}_{s,\iota}^{\#,1}$.

    \emph{Step 2}. If $K^{\iota} \subseteq F_{s,\iota}^1$ but $K^{\iota} \nsubseteq F_{s,\iota}^2$, we assign $K$ to $\mathcal{I}_{s,\iota}^{\#,2}$.

     \emph{Step 3}. For cubes with $K^{\iota} \subseteq F_{s,\iota}^2$, we continue this process recursively.

It is then clear that the family $\{\mathcal{I}_{s,\iota}^{\#,n} : n = 1, 2, \dots\}$ indeed partitions the relevant cubes $K \in \mathfrak{D}^{\overrightarrow{w}}$ in \eqref{e:27mainnon}. Therefore, together with \eqref{e:27supp}, we obtain
\Be\label{e:27mainnondecom}
\begin{split}
\sum_{K \in \mathfrak{D}^{\overrightarrow{w}}} T_{K} b_{k-s}(x)&=\sum_{\iota=1}^{2^d}\sum_{K \in \mathfrak{D}^{\overrightarrow{w}}} T_{K} b_{k-s}(x)\chi_{K^{\iota}}(x)\\
&=\sum_{\iota=1}^{2^d}
\sum_{n\geq1}\sum_{K \in \mathcal{I}_{s,\iota}^{\#,n}} T_{K} b_{k-s}(x)\chi_{K^{\iota}}(x).
\end{split}
\Ee

Below let us present a key result for the study of Proposition \ref{l:l3}.
\begin{proposition}\label{l:27exclude}
There exists a positive constant $\delta$ such that for all $n\geq1$ and $1\leq \iota\leq 2^d$ the following estimate holds
  \begin{equation}\label{e:e7}
    \Big\|\sup_{l\in\Z} \big|\sum_{K \in \mathcal{I}_{s,\iota}^{\#,n} \atop l(K) \geq 2^l}(T_K b_{k-s})\chi_{K^{\iota}}\big| \Big\|_{L^2(\mathbb{R}^d)} \lc 2^{-n}\big(s^22^{-\delta{s}}\lam\C_\Om\|f\|_{L^1(\mathbb{R}^d)}\big)^{\fr12}.
  \end{equation}
\end{proposition}

The proof of Proposition \ref{l:27exclude} will be given later. We first use Proposition \ref{l:27exclude} to finish the proof of Proposition \ref{l:l3}.

\vskip0.24cm
\subsection{Proof of Proposition \ref{l:l3}}\label{s:2734}
From the equality \eqref{e:27mainnondecom} and the triangle inequality, we derive
\Bes
\begin{split}
\sup_{l\in\Z}\big|\sum_{K \in \mathfrak{D}^{\overrightarrow{w}}\atop l(K)\geq2^l} T_{K} b_{k-s}(x)\big|&=
\sup_{l\in\Z}\big|\sum_{\iota=1}^{2^d}
\sum_{n\geq1}\sum_{K \in \mathcal{I}_{s,\iota}^{\#,n}\atop l(K)\geq2^l} T_{K} b_{k-s}(x)\chi_{K^\iota}(x)\big|\\
&\leq
\sum_{\iota=1}^{2^d}\sum_{n\geq1}\sup_{l\in\Z}\big|\sum_{K \in \mathcal{I}_{s,\iota}^{\#,n}\atop l(K)\geq2^l} T_{K} b_{k-s}(x)\chi_{K^\iota}(x)\big|.
\end{split}
\Ees
Hence,  by the triangle inequality and Proposition \ref{l:27exclude} we conclude that
  \begin{align*}
    \Big\|\sup_{l\in\Z}\big|\sum_{K \in \mathfrak{D}^{\overrightarrow{w}}\atop l(K)\geq2^l}T_K b_{k-s}\big|\Big\|_{L^2(\mathbb{R}^d)}
    & \leq \sum_{\iota=1}^{2^d}\sum_{n=1}^{\infty} \Big\|\sup_{l\in\Z}\big|\sum_{K \in \mathcal{I}_{s,\iota}^{\#, n} \atop  l(K)\geq2^l}(T_K b_{k-s})\chi_{K^\iota}\big|\Big\|_{L^2(\mathbb{R}^d)} \\
    & \lesssim \big(s^22^{-\delta s}\lam\C_\Omega\|f\|_{L^1(\mathbb{R}^d)}\big)^{\fr12},
  \end{align*}
  which completes the proof of Proposition \ref{l:l3}.
$\hfill{} \Box$
\vskip 0.24cm

\subsection{Proof of Proposition \ref{l:27exclude}: Linearizing the maximal operator}\label{s:2735}\quad
\vskip 0.24cm

The crucial part of the proof is to linearize the maximal operator.    Our strategy in this proof is to linearize the $L^2$ norm for the maximal operator in \eqref{e:e7}. The Rademacher-Menshov theorem plays a key role here.

Set $u_0 = C_02^{ds}$,  where $C_0$ is the constant in Definition \ref{d:27Fsn}.
Recall that for $n\geq1$, $\mathcal{I}_{s,\iota}^{\#,n}$ is defined by the following successive technique
$$\mathcal{I}_{s,\iota}^{\#, n} = \{K \in \mathfrak{D}^{\overrightarrow{w}}: b_{k-s}\chi_{\frac{1}{2}K}\neq0,  K^\iota\nsubseteq F_{s,\iota}^n, K^\iota \subseteq F_{s,\iota}^{n-1}\},$$
where we exclude all dyadic cubes $K$ such that $K^\iota\subseteq F_{s,\iota}^n$. This means that for each $K\in\mathcal{I}_{s,\iota}^{\#,n}$, the interior of $K^\iota$ is contained in $(F_{s,\iota}^n)^c$ or intersects both $F_{s,\iota}^n$ and $(F_{s,\iota}^n)^c$.
Observe that for each cube $K\in\mathcal{I}_{s,\iota}^{\#,n}$, $K^\iota$ has at most $u_0$ ancient dyadic cubes in $F_{s,\iota}^{n-1}$. Otherwise $K^\iota$ has more than $u_0$ ancient dyadic cubes contained in $F_{s,\iota}^{n-1}$, then $K^\iota\subseteq F_{s,\iota}^n$ which is a contradiction. If we set $\mathcal{L}_{s,\iota}^{\#,n}=\{K^\iota:K\in\mathcal{I}_{s,\iota}^{\#,n}\}$, then any two dyadic cubes in $\mathcal{L}_{s,\iota}^{\#,n}$ satisfy the property that either they are disjoint or one contains the other.

Let $\mathcal{L}_1$ denote the collection of the maximal elements of $\mathcal{L}_{s,\iota}^{\#,n}$. Define $\mathcal{L}_2$ as the collection of  the maximal elements of $\mathcal{L}_{s,\iota}^{\#,n}\backslash\mathcal{L}_1$. Proceeding inductively, we set
\[
\mathcal{L}_{u+1} =\{ \text{maximal elements of } \mathcal{L}_{s,\iota}^{\#,n}\backslash\cup_{v=1}^u \mathcal{L}_v\}.
\] See for example in Figure \ref{f:27} how we select  $\mathcal{L}_{1}, \mathcal{L}_2, \cdots,\mathcal{L}_{u+1},\cdots$.
Observe that these dyadic cubes within each $\mathcal{L}_v$ are pairwise disjoint.

By the definition of $\mathcal{L}_{s,\iota}^{\#,n}$ and each $K^\iota\in\mathcal{L}_{s,\iota}^{\#,n}$ has at most $u_0$ ancient cubes in $\mathcal{L}_{s,\iota}^{\#,n}$,  there are most $u_0$ generations for each dyadic cube  $K^\iota\in \mathcal{L}_1$ in $\mathcal{L}_{s,\iota}^{\#,n}$.  Therefore the induction construction argument of $\mathcal{L}_1, \ldots, \mathcal{L}_u$ will
  stop for $u\geq u_0+1$,  i.e.   $\mathcal{L}_{u}=\emptyset$ for $u\geq u_0+1$.

\begin{figure}[!htbp]
\centering
\includegraphics[height=0.18\textwidth]{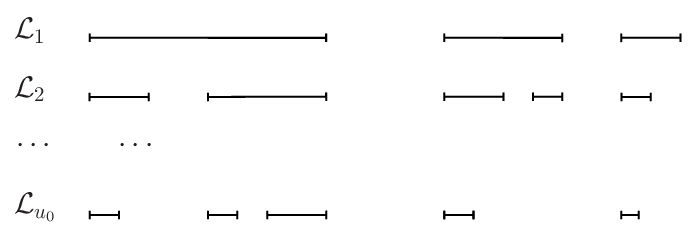}
\caption{\small We illustrate above the selection algorithm for the collections $\mathcal{L}_{u}$ in the one-dimensional case for simplicity. Take $\mathcal{L}_1$ to be the maximal dyadic cubes in $\mathcal{L}_{s,\iota}^{\#,n}$ (containing three cubes in this example). Define $\mathcal{L}_2$ as the maximal elements of $\mathcal{L}_{s,\iota}^{\#,n}\backslash\mathcal{L}_1$ (five cubes in this case). For each $u \geq 1$, set $\mathcal{L}_{u+1}$ to be the maximal elements of $\mathcal{L}_{s,\iota}^{\#,n}\backslash\cup_{v=1}^u\mathcal{L}_v$.}
\label{f:27}
\end{figure}

For $u=1,\cdots, u_0$, set $\mathcal{M}_u=\{K: K^\iota\in\mathcal{L}_u\}$ and
$$\beta_u \triangleq \sum_{K\in \mathcal{M}_u} (T_K b_{k-s})\chi_{K^\iota}.$$

We claim the following orthogonal estimate:
For any choice of $\epsilon_u \in \{-1, 1\}$ where $1\leq u\leq u_0$,  there exists a positive constant $\del$ such that
  \begin{equation}\label{e:e11}
    \Big\|\sum_{u=1}^{u_0} \epsilon_u \beta_u \Big\|_{L^2(\mathbb{R}^d)} \lesssim 2^{-n}\big(2^{-\delta{s}}\lam\C_\Om\|f\|_{L^1(\mathbb{R}^d)}\big)^{\fr12}.
  \end{equation}
The proof of \eqref{e:e11} is postponed in the next section and instead we provide the proof of \eqref{e:e7} first.   By (\ref{e:e11}) together with the Rademacher-Menshov theorem in Lemma \ref{l:l4},  we get that
  \begin{equation}\label{e:e12}
    \Big\|\sup_{0<v\leq  u_0} \big|\sum_{u=1}^{v} \beta_u\big| \Big\|_{L^2(\mathbb{R}^d)} \lesssim s2^{-n}\big(2^{-\delta{s}}\lam\C_\Om\|f\|_{L^1(\mathbb{R}^d)}\big)^{\fr12}.
  \end{equation}

We now show that the maximal function in (\ref{e:e7}) coincides exactly with the one in (\ref{e:e12}).
Applying
$$\beta_u = \sum_{K \in \mathcal{M}_u}(T_K b_{k-s})\chi_{K^\iota}= \sum_{K^\iota \in \mathcal{L}_u}(T_K b_{k-s})\chi_{K^\iota},$$
and these cubes $K^\iota\in\mathcal{L}_u$ in the above sum are disjoint for each $u$
together with $\supp[(T_K b_{k-s})\chi_{K^\iota}] \subseteq K^\iota$,  we obtain a crucial property of the following nesting relation for cubes in $\mathcal{L}_u$:   For any $J^\iota \in \mathcal{L}_u$ and $K^\iota\in \mathcal{L}_v$ with $u < v$, we have either $J^\iota\cap K^{\iota} = \emptyset$ or $K^\iota \subseteq J^\iota$. This implies that for each $x$ in the support  of
  $$\sum\limits_{K \in \mathcal{I}_{s,\iota}^{\#,n}} (T_K b_{k-s})\chi_{K^\iota}$$
and each $0<u\leq u_0$, either there exists no dyadic cube $K^\iota \in \mathcal{L}_u$ such that   $x\in K^\iota$ or there exists exactly one dyadic cube $K^{\iota} \in \mathcal{L}_u$ containing $x$.
In the latter case, we denote this unique cube by $K^{\iota}_{u,x}$ and its father cube in $\mathfrak{D}^{\overrightarrow{w}}$ by $K_{u,x}$.
Then these cubes $\{K^{\iota}_{u,x}\}_{u}$ naturally form a nested sequence  containing $x$.

As illustrated in Figure~\ref{f:27}, we consider a point $x$ contained in the third cube of $\mathcal{L}_{u_0}$. Then the corresponding cubes from $\mathcal{L}_1$ (the first one), $\mathcal{L}_2$ (the second one), $\cdots$ and $\mathcal{L}_{u_0}$ (the third one) form a sequence of dyadic cubes containing $x$. Consequently, it suffices to evaluate the maximal function over this nested sequence of dyadic cubes.
Therefore we obtain
  \begin{align*}
    \sup_{l \in\Z} \Big|\sum_{K \in \mathcal{I}_{s,\iota}^{\#,n};2^l \leq l(K)}& T_K b_{k-s}(x)\chi_{K^\iota}(x)\Big|
     = \sup_{1\leq v\leq u_0}\Big|\sum_{u=1}^v T_{K_{u,x}}b_{k-s}(x)\chi_{K^\iota}(x)\Big| \\
    &= \sup_{1\leq v\leq u_0} \Big|\sum_{u=1}^{v} \sum_{K \in \mathcal{M}_u} T_K b_{k-s} (x)\chi_{K^\iota}(x)\Big| = \sup_{1\leq v\leq u_0} \Big|\sum_{u=1}^{v}  \beta_u(x)\Big|.
  \end{align*}
By the above discussion, we get (\ref{e:e7}) from (\ref{e:e12}). Hence we complete the proof of Proposition \ref{l:27exclude} based on claim \eqref{e:e11}.
$\hfill{} \Box$
\vskip 0.24cm

\section{Proof of Proposition \ref{l:27exclude}: estimate for claim \eqref{e:e11} of linearized operator}\label{s:274}

In this section, we present the proof of claim \eqref{e:e11}. First, we express the operator $\sum_{u=1}^{u_0} \epsilon_u \beta_u$ in a more explicit form and reformulate claim \eqref{e:e11} as Proposition \ref{l:27finall2} below.
We then decompose this explicit operator  into four parts and reduce the proof of Proposition \ref{l:27finall2} to establishing four key lemmas (Lemma \ref{l:27finall21}, Lemma \ref{l:27finall22}, Lemma \ref{l:27finall23}, and Lemma \ref{l:27finall24}). The remainder of this section is devoted to  the proofs of Lemma \ref{l:27finall21}, Lemma \ref{l:27finall22}, and Lemma \ref{l:27finall23}. Before that, we will also give some $L^3$ trivial estimates for the decomposed operators. The proof of Lemma \ref{l:27finall24}, being more involved, is deferred to Section \ref{s:275}.
\vskip0.24cm

\subsection{Restatement of claim \eqref{e:e11}}
Write
  \begin{align*}
    \Big\|\sum_{u=1}^{u_0} \epsilon_u \beta_u\Big\|_{L^2(\mathbb{R}^d)}
    & = \Big\|\sum_{u=1}^{u_0} \epsilon_u \sum_{K \in \mathcal{M}_u}(T_K b_{k-s})\chi_{K^\iota}\Big\|_{L^2(\mathbb{R}^d)} \\
    & = \Big\|\sum_{j \in \mathbb{Z}} \sum_{J\in \mathcal{I}_{s,\iota}^{\#,n} \atop l(J)=2^j} \epsilon_J (T_J b_{j-s})\chi_{J^\iota}\Big\|_{L^2(\mathbb{R}^d)},
  \end{align*}
where we make a change of notions $K\mapsto J,\ k\mapsto j$ and set $\epsilon_J=\epsilon_u$ for $J\in\mathcal{M}_u$.
Recall that the kernel of $T_j$ is $\mathcal{K}_j$.  For a fixed $j$,  we have
\Be\label{e:27TJQ1}
  \begin{split}
    \sum_{J \in \mathcal{I}_{s,\iota}^{\#,n} \atop l(J)=2^j}\epsilon_J T_J b_{j-s}(x)\chi_{J^\iota}(x) &= \sum_{{J \in \mathcal{I}_{s,\iota}^{\#,n} \atop l(J)=2^j}}\mathcal{K}_j\ast ( \epsilon_J b_{j-s}\chi_{\frac{1}{2}J})(x)\chi_{J^\iota}(x)\\
    &=\sum_{{J \in \mathcal{I}_{s,\iota}^{\#,n} \atop l(J)=2^j}}\sum_{Q\in\mathcal{Q}_{j-s}} \mathcal{K}_j\ast (\epsilon_J b_{Q}\chi_{\frac{1}{2}J})(x)\chi_{J^\iota}(x).
  \end{split}
\Ee
The cancelation property of $b_Q$ is important for our later proof. Therefore we need to remove the characteristic function $\chi_{\frac{1}{2}J}$ associated to the bad function $b_Q$.  Notice that we only consider these dyadic cubes $Q\in\mathcal{Q}_{j-s}$ such that $Q\cap\fr12J\neq\emptyset$. Since $s$ is larger than 200, $Q\in\mathfrak{D}$ with $l(Q)=2^{j-s}$ and $J\in\mathfrak{D}^{\overrightarrow{w}}$ with $l(J)=2^j$, we  have  $Q \subseteq \frac{1}{2}J$ in view  of a simple geometry observation.
Hence we get
  \begin{align}\label{e:27TJQ}
    \sum_{J \in \mathcal{I}_{s,\iota}^{\#,n} \atop l(J)=2^j}\epsilon_J T_J b_{j-s}(x)\chi_{J^\iota}(x)  = \sum_{J\in \mathcal{I}_{s,\iota}^{\#,n} \atop l(J)=2^j} \sum_{ Q \in  \mathcal{Q}_{j-s} \atop  Q\subseteq \frac{1}{2}J} \epsilon_J T_j b_{Q}(x)\chi_{J^\iota}(x) .
  \end{align}
To simply our notation, let $\mathfrak{Q}$ be given by
\Bes
\mathfrak{Q}=\bigcup_{j\in\Z}\mathfrak{Q}_{j-s}
\Ees
where
\Bes
\mathfrak{Q}_{j-s}=\Big\{Q\in \mathcal{Q}_{j-s}:\ Q\subseteq\fr12 J\ \text{where}\ J \in \mathcal{I}_{s,\iota}^{\#,n},\ l(J)=2^j\Big\}.
\Ees

Notice that by the Calder\'on-Zygmund decomposition property (cz-iv),  $\mathfrak{Q}$ is a collection of  dyadic cubes with  disjoint interiors and $\mathfrak{Q}_{j-s}=\{Q\in\mathfrak{Q}:l(Q)=2^{j-s}\}$. Hence the right side of \eqref{e:27TJQ} equals to
\Be\label{e:27TJQ2}
\sum_{ Q \in  \mathcal{Q}_{j-s}}\sum_{J\in \mathcal{I}_{s,\iota}^{\#,n},l(J)=2^j\atop  Q\subseteq \frac{1}{2}J}\epsilon_J T_j b_{Q}(x)\chi_{J^\iota}(x) = \sum_{Q\in\mathfrak{Q}_{j-s}}\epsilon_J T_j b_{Q}(x)\chi_{J^\iota}(x),
\Ee
where the above equality follows from that for a fixed $Q\in\mathcal{Q}_{j-s}$, there exists at most one dyadic cube $J\in\mathcal{I}_{s,\iota}^{\#,n}$ with fixed sidelength $l(J)=2^j$ such that $Q\subseteq\frac{1}{2}J$. Moreover for a fixed $Q\in\mathfrak{Q}_{j-s}$, $J$ is determined by $Q$ so we get $\epsilon_J$ and $\chi_{J^\iota}$ are well defined.
From \eqref{e:27TJQ1}, \eqref{e:27TJQ} and \eqref{e:27TJQ2}, we conclude that
\begin{align*}
    \sum_{J \in \mathcal{I}_{s,\iota}^{\#,n} \atop l(J)=2^j}\epsilon_J T_J b_{j-s}(x)\chi_{J^\iota}(x) &=\sum_{Q\in\mathfrak{Q}_{j-s}}\epsilon_J T_j b_{Q}(x)\chi_{J^\iota}(x).
\end{align*}

Here we also present some useful observations for the sum of $\|b_Q\|_{L^1(\R^d)}$.
Recall $F_{s,\iota}^{n-1}$ can be rewritten as a union of the maximal cubes: $F_{s,\iota}^{n-1}=\bigcup_{J^\iota\in\mathcal{Q}_{s,\iota,n-1}}J^\iota$. If setting $\tilde{F}_{s,\iota}^{n-1}=\bigcup_{J^\iota\in\mathcal{Q}_{s,\iota,n-1}}J$, then in the proof of Lemma \ref{l:27Fsn} we have shown that $|\tilde{F}_{s,\iota}^{n-1}|\leq 2^{d}|{F}_{s,\iota}^{n-1}|$. Therefore
\Be\label{e:27tildeFn}
\Big|\bigcup_{J\in \mathcal{I}_{s,\iota}^{\#,n}}J\Big|\leq \Big|\bigcup_{J^\iota\in\mathcal{Q}_{s,\iota,n-1}}J|=|\tilde{F}_{s,\iota}^{n-1}|\leq 2^{d}|{F}_{s,\iota}^{n-1}|
\Ee
where in the first inequality we use the fact every $J\in\mathcal{I}_{s,\iota}^{\#,n}$ satisfies $J^\iota\subseteq F_{s,\iota}^{n-1}$ and $F_{s,\iota}^{n-1}$ can be rewritten as a union of the disjoint cubes $F_{s,\iota}^{n-1}=\bigcup_{J^\iota\in\mathcal{Q}_{s,\iota,n-1}}J^\iota$.
By the property (cz-iii) in the Calder\'on-Zygmund decomposition,  we obtain for all $n\geq1$
  \begin{equation}\label{e:27sumbQ}
  \begin{split}
      \sum_{Q\in\mathfrak{Q}}\|b_Q\|_{L^1(\mathbb{R}^d)}
    &\lesssim \lam\C_\Omega^{-1}\sum_{j\in\Z}\sum_{ Q\in \mathfrak{Q}_{j-s}}| Q|\\
    &\leq\lam\C_\Omega^{-1}\min\Big\{\Big|\bigcup_{Q\in\mathcal{Q}}Q\Big|, \Big|\bigcup_{J\in \mathcal{I}_{s,\iota}^{\#,n}}J\Big|\Big\}\\
    &\lc\lam\C_\Omega^{-1}\min\Big\{\Big|\bigcup_{Q\in\mathcal{Q}}Q\Big|,|F_{s,\iota}^{n-1}|\Big\}\\
    &\lc 2^{-2n}\|f\|_{L^1(\mathbb{R}^d)}
  \end{split}
  \end{equation}
where the second inequality follows from that for a fixed $Q\in\mathfrak{Q}_{j-s}$, there exists unique dyadic cube $J\in\mathcal{I}_{s,\iota}^{\#,n}$ with fixed sidelength $l(J)=2^j$ such that $Q\subseteq\frac{1}{2}J$ and  all dyadic cubes in $\mathfrak{Q}$ are disjoint,  while in the last two inequalities we use  the estimate \eqref{e:27tildeFn}, the property (cz-v) and Lemma \ref{l:27Fsn}.

Now we can restate the claim \eqref{e:e11} as follows.

\begin{proposition}\label{l:27finall2}
With all notions above, there exists a positive constant $\del$ such that
  \begin{equation*}
    \Big\|\sum_{j\in\Z}\sum_{Q\in\mathfrak{Q}_{j-s}}\epsilon_J (T_j b_{Q})\chi_{J^\iota}\Big\|_{L^2(\mathbb{R}^d)} \lesssim 2^{-n}\big(2^{-\delta{s}}\lam\C_\Om\|f\|_{L^1(\mathbb{R}^d)}\big)^{\fr12}.
  \end{equation*}
\end{proposition}

\vskip0.24cm

\subsection{Some approximations and microlocal decomposition}\label{s:27submic}\quad
\vskip0.24cm

Since $\chi_{J^\iota}$ is a characteristic function which is not smooth at the boundary $\pari(J^\iota)$, we need to make a smooth approximation of $\chi_{J^\iota}$. Let $\varpi$ be a nonnegative, radial $C_c^\infty$ function which is supported in $\{|x|\leq2^{-5}\}$ and $\int_{\R^d}\varpi(x)dx=1$. Set $\varpi_j(x)=2^{-jd}\varpi(2^{-j}x)$. Define the operator $P_j$ by
$$P_jg(x)=\int_{\R^d}\varpi_{j}(x-z)g(z)dz.$$
In particular, define $\tilde{\chi}_{J^\iota}(x)=P_{j-s\varkappa}[\chi_{J^\iota}](x)$
where $\varkappa\in(0,1)$ is a small constant to be chosen later. By the definition of $\tilde{\chi}_{J^\iota}(x)$, we have
\Be\label{e:27parix}
|\pari_x^\alp \tilde{\chi}_{J^\iota}(x)|\lc 2^{-(j-s\varkappa)|\alp|}.
\Ee

\medskip

Next, to deal with the rough kernel, we will employ the modified microlocal decomposition method from Seeger \cite{See96}.
We begin by constructing a microlocal decomposition of the kernel. Fix an integer $s\geq200$ and consider a collection of unit vectors $\Theta_s=\{e^s_v\}_v$  satisfying:

(a)\ Separation condition: $|e^s_v-e^s_{v'}|\geq 2^{-s\ga-4}$  if $v\neq v'$;

(b)\ Covering property: If $\theta\in S^{d-1}$, there exists an $e^s_v$ such that $|e^s_v-\theta|\leq 2^{-s\ga-4}$.

Here the constant $\ga$ in (a) and (b) satisfies $0<\ga<\varkappa<1$  which will be specified later. Such a collection $\Theta_s$ can be obtained by taking a maximal set satisfying condition (a) and then condition (b) will hold automatically. Notice that there are $C2^{s\ga(d-1)}$
elements in the collection $\{e^s_v\}_v$. For every $\tet\in S^{d-1}$,
there only exists finite $e^s_v$ such that $|e^s_v-\tet|\leq2^{-s\gamma-4}$.
Therefore we can construct disjoint measurable sets $E^s_v\subseteq S^{d-1}$ such that $e^s_v\in E^s_v$, $\diam(E^s_v)\leq2^{-s\gamma-2}$ and   $\bigcup_{v}E^{s}_v=S^{d-1}$.
We note that, in contrast to the original microlocal decomposition in \cite{See96},  an additional parameter $\gamma$ is introduced here. This parameter will be chosen sufficiently small and play a crucial role in the subsequent interpolation argument.

\medskip

Recall that the kernel of $T_j$ is $\mathcal{K}_j$. Based the above microlocal decomposition, we define an operator $T_j^{v}$ by
\Be\label{e:27Def T_j^{s,v}}
T_j^{v}g(x)=\int_{\R^d}\K_j^{v}(x-y)g(y)dy
\Ee
where $\K_{j}^v(x)=\K_j(x)\chi_{E^s_v}(x/|x|)$.
Hence we have
$T_j=\sum\limits_vT_j^{v}.$
In the frequency space  we need to separate the phase into different directions so we define a Fourier multiplier operator by
$$\widehat{G^s_{v}g}(\xi)=\Phi(2^{s\gamma}\inn{e^s_v}{{\xi}/{|\xi|}})\hat{g}(\xi),$$
where $\Phi$ is a smooth, nonnegative, radial function such that
$0\leq\Phi(x)\leq1$ and $\Phi(x)=1$ on $|x|\leq2$, $\Phi(x)=0$ on $|x|>4$.

In the following, we split $\epsilon_J (T_j b_{Q})\chi_{J^\iota}$ into four parts:
\begin{equation*}
\begin{split}
  \epsilon_J (T_j b_{Q})\chi_{J^\iota}&=\epsilon_J (T_j b_{Q})(\chi_{J^\iota}-\tilde{\chi}_{J^\iota})+P_{j-s\varkappa}[ \epsilon_J(T_j b_{Q})\tilde{\chi}_{J^\iota}]\\
&\quad+\sum_v(I-P_{j-s\varkappa})G^s_v[\epsilon_J(T^v_j b_{Q})\tilde{\chi}_{J^\iota}]\\
&\quad+\sum_v(I-P_{j-s\varkappa})(I-G^s_v)[\epsilon_J(T^v_j b_{Q})\tilde{\chi}_{J^\iota}].
\end{split}
\end{equation*}

Consequently, Proposition \ref{l:27finall2} follows directly from the subsequent four lemmas.

\begin{lemma}\label{l:27finall21}
With all notions above, there exists a positive constant $\del_1$ such that
  \begin{equation*}
    \Big\|\sum_{j\in\Z}\sum_{Q\in\mathfrak{Q}_{j-s}}\epsilon_J (T_j b_{Q})(\chi_{J^\iota}-\tilde{\chi}_{J^\iota})\Big\|_{L^2(\mathbb{R}^d)} \lesssim 2^{-n}\big(2^{-\delta_1{s}}\lam\C_\Om\|f\|_{L^1(\mathbb{R}^d)}\big)^{\fr12}.
  \end{equation*}
\end{lemma}

\begin{lemma}\label{l:27finall22}
With all notions above, there exists a positive constant $\del_2$ such that
  \begin{equation*}
    \Big\|\sum_{j\in\Z}\sum_{Q\in\mathfrak{Q}_{j-s}}P_{j-s\varkappa}[ \epsilon_J(T_j b_{Q})\tilde{\chi}_{J^\iota}]\Big\|_{L^2(\mathbb{R}^d)} \lesssim 2^{-n}\big(2^{-\delta_2{s}}\lam\C_\Om\|f\|_{L^1(\mathbb{R}^d)}\big)^{\fr12}.
  \end{equation*}
\end{lemma}

\begin{lemma}\label{l:27finall23}
With all notions above, there exists a positive constant $\del_3$ such that
  \begin{equation*}
    \Big\|\sum_{j\in\Z}\sum_{Q\in\mathfrak{Q}_{j-s}}\sum_v(I-P_{j-s\varkappa})G^s_v[\epsilon_J(T^v_j b_{Q})\tilde{\chi}_{J^\iota}]
\Big\|_{L^2(\mathbb{R}^d)} \lesssim 2^{-n}\big(2^{-\delta_3{s}}\lam\C_\Om\|f\|_{L^1(\mathbb{R}^d)}\big)^{\fr12}.
  \end{equation*}
\end{lemma}

\begin{lemma}\label{l:27finall24}
With all notions above, there exists a positive constant $\del_4$ such that
  \begin{equation*}
  \begin{split}
    \Big\|\sum_{j\in\Z}\sum_{Q\in\mathfrak{Q}_{j-s}}\sum_v(I-P_{j-s\varkappa})(I-G^s_v)&[\epsilon_J(T^v_j b_{Q})\tilde{\chi}_{J^\iota}]
\Big\|_{L^2(\mathbb{R}^d)}\\
& \lesssim 2^{-n}\big(2^{-\delta_4{s}}\lam\C_\Om\|f\|_{L^1(\mathbb{R}^d)}\big)^{\fr12}.
  \end{split}\end{equation*}
\end{lemma}

We first give the proofs of Lemma \ref{l:27finall21}, Lemma \ref{l:27finall22} and Lemma \ref{l:27finall23} in this section. The proof of Lemma \ref{l:27finall24} is very long and will be presented in the next section.

Notice we need to establish some $L^2$ estimates with exponential bounds for Lemma \ref{l:27finall21}, Lemma \ref{l:27finall22}, Lemma \ref{l:27finall23} and Lemma \ref{l:27finall24}. Except Lemma \ref{l:27finall23}, both our strategies are to make an interpolation between an $L^1$ estimate with a nice  decay bound and an $L^3$ estimate with a trivial bound.

\vskip0.24cm

\subsection{Some trivial bounds for $L^3$ estimates} In this subsection, we first give $L^3$ estimates for these operators appear in  Lemma \ref{l:27finall21}, Lemma \ref{l:27finall22} and Lemma \ref{l:27finall24}.

\begin{lemma}\label{l:27L^31}
With all notions above, we have
\Bes
\begin{split}
\Big\|\sum_j\sum_{Q\in\mathfrak{Q}_{j-s}}|T_jb_Q|\Big\|_{L^3(\R^d)}\lc\|\Om\|_\infty\Big(\lambda^2\C_\Omega^{-2}\sum_{Q\in\mathfrak{Q}}\|b_Q\|_{L^1(\R^d)}\Big)^{1/3}.
\end{split}\Ees
  \end{lemma}
\begin{proof}
Let $\mathfrak{B}_{j-s}=\sum_{Q\in\mathfrak{Q}_{j-s}}|b_Q|$. Rewrite
\Be\label{e:27sumv}
\begin{split}
\Big\|\sum_j\sum_{Q\in\mathfrak{Q}_{j-s}}|T_jb_Q|\Big\|^3_{L^3(\R^d)}
&=\int\Big|\sum_j|\mathcal{K}_j|*\mathfrak{B}_{j-s}(x)\Big|^3dx\\
&\leq 3!\sum_{j_1\geq j_2\geq j_3}\int\prod_{t=1}^{3}|\mathcal{K}_{j_t}|*\mathfrak{B}_{j_t-s}(x)dx\\
&\lc\sum_{j_1\geq j_2\geq j_3}\int\prod_{t=1}^{3}\int|\mathcal{K}_{j_t}(x-y_t)\mathfrak{B}_{j_t-s}(y_t)|dy_t dx.
\end{split}\Ee
By changing the order of integration, the last term above equals to
\Be\label{e:27prod}
\sum_{j_1\geq j_2\geq j_3}\iiint\Big[\int\prod_{t=1}^{3}|\mathcal{K}_{j_t}(x-y_t)|dx\Big]
\mathfrak{B}_{j_1-s}(y_1)\mathfrak{B}_{j_2-s}(y_2)\mathfrak{B}_{j_3-s}(y_3)dy_3dy_2dy_1.
\Ee

Using the supports of $\mathcal{K}_{j_t}(x-y_t)$ for $t=1,2,3$, we get $|y_3-y_2|\leq|x-y_2|+|x-y_3|\lc2^{j_2}+2^{j_3}\lc2^{j_2}$ and $|y_2-y_1|\lc2^{j_1}$. So it is easy to see that
\Bes
|\mathcal{K}_{j_1}(x-y_1)|\lc\|\Omega\|_\infty2^{-j_1d}\chi_{\{|y_1-y_2|\lc2^{j_1}\}};
\Ees
\Bes
|\mathcal{K}_{j_2}(x-y_2)|\lc\|\Omega\|_\infty2^{-j_2d}\chi_{\{|y_2-y_3|\lc2^{j_2}\}}
\Ees
and
\Bes
\int\prod_{t=1}^{3}|\mathcal{K}_{j_t}(x-y_t)|dx\lc\|\Omega\|^3_\infty2^{-(j_2+j_1)d}\chi_{\{|y_1-y_2|\lc2^{j_1}\}}\chi_{\{|y_2-y_3|\lc2^{j_2}\}}.
\Ees
Therefore by the preceding inequality,
we obtain that \eqref{e:27prod} is majorized by
\Be\label{e:27prodwithout}
\begin{split}
\|\Omega\|^3_\infty\sum_{j_1\geq j_2\geq j_3}2^{-(j_2+j_1)d}&\iiint\chi_{\{|y_1-y_2|\lc2^{j_1}\}}\chi_{\{|y_2-y_3|\lc2^{j_2}\}}\\
&\times
\mathfrak{B}_{j_1-s}(y_1)\mathfrak{B}_{j_2-s}(y_2)\mathfrak{B}_{j_3-s}(y_3)dy_3dy_2dy_1.
\end{split}\Ee

We first consider the integration over $y_3$ and the sum over $j_3$ in \eqref{e:27prodwithout} while fix $y_2,y_1$ and $j_2,j_1$. Then we derive
\Be\label{e:27K1ij}
\begin{split}
\sum_{j_3\leq j_2}\int_{|y_2-y_3|\leq2^{j_2}}\mathfrak{B}_{j_3-s}(y_3)dy_3
&=\sum_{j\leq j_2}\int_{|y_2-y|\lc2^{j_2}}\mathfrak{B}_{j-s}(y)dy\\
&\leq\sum_{j\leq j_2} \sum_{ Q\in \mathfrak{Q}_{j-s}} \int_{|y_2-y| \lesssim2^{j_2}}|b_ Q(y)|dy\\
&\lc\lam\C_\Om^{-1}\sum_{j\leq j_2} \sum_{ Q\in \mathfrak{Q}_{j-s}\atop dist(y_2,Q)\lc2^{j_2}} |Q|
\end{split}
\Ee
where in the first equality  we just make a change of variables $j_3\mapsto j$ and $y_3\mapsto y$, while  the last inequality follows from our assumption  $b_Q$ is supported in $Q$ and $\|b_Q\|_{L^1(\mathbb{R}^d)} \leq \lam\C_\Om^{-1}| Q|$.  Notice that $j\leq j_2$ and all cubes in $\mathfrak{Q}$ are disjoint, we obtain
  \begin{equation}\label{e:e15}
  \sum_{j\leq j_2}\sum_{ Q \in \mathfrak{Q}_{j-s} \atop dist (y_2,  Q)\lesssim 2^{j_2}} | Q| \lesssim 2^{dj_2}.
  \end{equation}
Therefore substituting \eqref{e:e15} into \eqref{e:27K1ij}, we get the following estimate
\Be\label{e:27j_3}
2^{-dj_2}\sum_{j_3\leq j_2}\int_{|y_2-y_3|\leq2^{j_2}}\mathfrak{B}_{j_3-s}(y_3)dy_3\lc\lam\C_\Om^{-1}.
\Ee

Next we consider the integration over $y_2$ and the sum over $j_2$ in \eqref{e:27prodwithout} while fix $y_1$ and $j_1$. Similar to the proof for \eqref{e:27j_3}, we could obtain
\Be\label{e:27j_2}
2^{-dj_1}\sum_{j_2\leq j_1}\int_{|y_1-y_2|\leq2^{j_1}}\mathfrak{B}_{j_2-s}(y_2)dy_2\lc\lam\C_\Om^{-1}.
\Ee

Finally we consider the integration over $y_1$ and the sum over $j_1$ in \eqref{e:27prodwithout}, then
\Be\label{e:27j_1}
\sum_{j_1\in\Z}\int\mathfrak{B}_{j_1-s}(y_1)dy_1\leq\sum_{Q\in\mathfrak{Q}}\|b_Q\|_{L^1(\R^d)}.
\Ee
Substituting these estimates \eqref{e:27j_1},
\eqref{e:27j_2} and \eqref{e:27j_3} into  \eqref{e:27prodwithout} together with \eqref{e:27prod} and \eqref{e:27sumv}, we get the desired $L^3$ estimate.
\end{proof}

\begin{lemma}\label{l:27L^32}
With all notions above, we have
\Bes
\begin{split}
\Big\|\sum_j\sum_{Q\in\mathfrak{Q}_{j-s}}P_{j-s\varkappa}[|T_jb_Q|]\Big\|_{L^3(\R^d)}\lc\|\Om\|_\infty\Big(\lambda^2\C_\Omega^{-2}\sum_{Q\in\mathfrak{Q}}\|b_Q\|_{L^1(\R^d)}\Big)^{1/3}.
\end{split}\Ees
  \end{lemma}
\begin{proof}
The proof is very similar to that of Lemma \ref{l:27L^31}. We only point out the difference here.
Notice that
$$
P_{j-s\varkappa}[|T_jb_Q|](x)\leq(P_{j-s\varkappa}|\mathcal{K}_j|)*|b_Q|(x)\lc H_j*|b_Q|(x)
$$
where $H_j(x)=2^{-jd}\chi_{\{2^{j-5}<|x|<2^{j-1}\}}\|\Om\|_\infty$. Now just replacing the kernel $|\mathcal{K}_j|$ in the proof of Lemma \ref{l:27finall21} by $H_j$ and proceeding the proof as we have done there, we could obtain the desired result for this lemma.
\end{proof}

\begin{lemma}\label{l:27L^33}
With all notions above, we have
\Bes
\begin{split}
    \Big\|\sum_{j\in\Z}\sum_{Q\in\mathfrak{Q}_{j-s}}\sum_v&(I-P_{j-s\varkappa})(I-G^s_v)[\epsilon_J(T^v_j b_{Q})\tilde{\chi}_{J^\iota}]
\Big\|_{L^3(\mathbb{R}^d)}\\
&\lc\|\Om\|_\infty2^{s\ga([\fr{d}{2}]+1)+\fr23s\gamma(d-1)}\Big(\lambda^2\C_\Omega^{-2}\sum_{Q\in\mathfrak{Q}}\|b_Q\|_{L^1(\R^d)}\Big)^{1/3}.
\end{split}\Ees
  \end{lemma}
\begin{proof}
By the triangle inequality, we get
\Bes
\begin{split}
\Big\|\sum_{j\in\Z}\sum_{Q\in\mathfrak{Q}_{j-s}}\sum_v(I-&P_{j-s\varkappa})(I-G^s_v)[\epsilon_J(T^v_j b_{Q})\tilde{\chi}_{J^\iota}]
\Big\|_{L^3(\mathbb{R}^d)}\\
&\leq\sum_v
\Big\|(I-G^s_v)\sum_{j\in\Z}\sum_{Q\in\mathfrak{Q}_{j-s}}(I-P_{j-s\varkappa})[\epsilon_J(T^v_j b_{Q})\tilde{\chi}_{J^\iota}]
\Big\|_{L^3(\mathbb{R}^d)}.
\end{split}
\Ees

Notice that the operator $(I-G^s_{v})$ is a Fourier multiplier operator with multiplier $1-\Phi(2^{s\ga}\inn{e^s_v}{\xi/|\xi|})$.
It is easy to see that $1-\Phi(2^{s\ga}\inn{e^s_v}{\xi/|\xi|})$ is bounded and
$$|\pari_{\xi}^{\alp}[1-\Phi(2^{s\ga}\inn{e^s_v}{\xi/|\xi|})]|\lc2^{s\ga([\fr{d}{2}]+1)}|\xi|^{-|\alp|}$$
for all multi indices $|\alp|\leq [\fr{d}{2}]+1$.
Then by Lemma \ref{l:27Hmihlin},  $I-G^s_{v}$ is of strong type $(3,3)$  with operator norm at most $C2^{s\ga([\fr{d}{2}]+1)}.$ Hence we get
\Be\label{e:27sumv3}
\begin{split}
\Big\|(I-G^s_v)\sum_{j\in\Z}&\sum_{Q\in\mathfrak{Q}_{j-s}}(I-P_{j-s\varkappa})[\epsilon_J(T^v_j b_{Q})\tilde{\chi}_{J^\iota}]
\Big\|_{L^3(\mathbb{R}^d)}\\
&\lc 2^{s\ga([\fr{d}{2}]+1)}\Big\|\sum_{j\in\Z}\sum_{Q\in\mathfrak{Q}_{j-s}}(I-P_{j-s\varkappa})[\epsilon_J(T^v_j b_{Q})\tilde{\chi}_{J^\iota}]
\Big\|_{L^3(\mathbb{R}^d)}.
\end{split}
\Ee

In what follows, we fix $v$. By the support of $\mathcal{K}_j^v$ and $\varpi_{j-s\varkappa}$ together with $0<\ga<\varkappa<1$, we get
\Be\label{e:27IPjs}
\begin{split}
|(I-P_{j-s\varkappa})[\epsilon_J(T^v_j b_{Q})\tilde{\chi}_{J^\iota}](x)|&\leq |\mathcal{K}_j^v|*|b_Q|(x)+(P_{j-s\varkappa}|\mathcal{K}_j^v|)*|b_Q|(x)\\
&\lc H_j^{s,v}*|b_Q|(x)
\end{split}
\Ee
where $H_j^{s,v}(x):=2^{-jd}\chi_{E^{s,v}_j}(x)\|\Om\|_\infty$ and $\chi_{E^{s,v}_j}(x)$ is a characteristic function of the set
$$E_j^{s,v}:=\{x\in \R^d:|\inn{x}{e^s_v}|\leq 2^{j-1},|x-\inn{x}{e^s_v}e^s_v|\leq 2^{j-1-s\gamma}\}.$$

Let $\mathfrak{B}_{j-s}=\sum_{Q\in\mathfrak{Q}_{j-s}}|b_Q|$. Rewrite
\Be\label{e:27prod3}
\begin{split}
&\quad\Big\|\sum_{j\in\Z}\sum_{Q\in\mathfrak{Q}_{j-s}}(I-P_{j-s\varkappa})[\epsilon_J(T^v_j b_{Q})\tilde{\chi}_{J^\iota}]
\Big\|^3_{L^3(\mathbb{R}^d)}\\
&\lc\int\Big|\sum_jH^{s,v}_j*\mathfrak{B}_{j-s}(x)\Big|^3dx
\leq 3!\sum_{j_1\geq j_2\geq j_3}\int\Big[\prod_{t=1}^{3} H^{s,v}_{j_t}*\mathfrak{B}_{j_t-s}(x)\Big]dx\\
&\lc\sum_{j_1\geq j_2\geq j_3}\int\Big[\prod_{t=1}^{3}\int H^{s,v}_{j_t}(x-y_t)\mathfrak{B}_{j_t-s}(y_t)dy_t\Big] dx.
\end{split}\Ee
By changing the order of integration, the last term above equals to
\Bes
\sum_{j_1\geq j_2\geq j_3}\iiint\Big[\int\prod_{t=1}^{3}H^{s,v}_{j_t}(x-y_t)dx\Big]
\mathfrak{B}_{j_1-s}(y_1)\mathfrak{B}_{j_2-s}(y_2)\mathfrak{B}_{j_3-s}(y_3)dy_3dy_2dy_1.
\Ees

Utilizing the supports of $H^{s,v}_{j_t}(x-y_t)$ for $t=1,2,3$, we get $|y_3-y_2|\lc2^{j_2}$ and $|y_2-y_1|\lc2^{j_1}$. So it is easy to see that
\Bes
\int\prod_{t=1}^{3}H^{s,v}_{j_t}(x-y_t)dx\lc\|\Omega\|^3_\infty2^{-s\gamma(d-1)}2^{-(j_2+j_1)d}\chi_{\{|y_1-y_2|\lc2^{j_1}\}}\chi_{\{|y_2-y_3|\lc2^{j_2}\}}.
\Ees
Therefore by the preceding inequality,
we obtain that \eqref{e:27prod3} is bounded by
\Bes
\begin{split}
\|\Omega\|^3_\infty2^{-s\gamma(d-1)}\sum_{j_1\geq j_2\geq j_3}2^{-(j_2+j_1)d}&\iiint\chi_{\{|y_1-y_2|\lc2^{j_1}\}}\chi_{\{|y_2-y_3|\lc2^{j_2}\}}\\
&\times\mathfrak{B}_{j_1-s}(y_1)\mathfrak{B}_{j_2-s}(y_2)\mathfrak{B}_{j_3-s}(y_3)dy_3dy_2dy_1.
\end{split}\Ees
Notice that this estimate is the same as \eqref{e:27prodwithout} in Lemma \ref{l:27L^31}. Combining the estimate for \eqref{e:27prodwithout}, together with \eqref{e:27prod3}, \eqref{e:27sumv3} and $\card(\Theta_s)\lc2^{s\ga(d-1)}$, we get the desired $L^3$ estimate.
\end{proof}
\vskip0.24cm

\subsection{Proof of Lemma \ref{l:27finall21}}
To establish Lemma \ref{l:27finall21}, the key is to show the following $L^1$ estimate with an exponential decay in $s$: For a fix $j\in\Z$ and $Q\in\mathfrak{Q}_{j-s}$,
\Be\label{e:27L1chi}
    \|(T_j b_{Q})(\chi_{J^\iota}-\tilde{\chi}_{J^\iota})\|_{L^1(\mathbb{R}^d)} \lesssim 2^{-s\varkappa}\|\Omega\|_\infty\|b_Q\|_{L^1(\R^d)}.
\Ee

We first prove \eqref{e:27L1chi}. Write $(T_j b_{Q})(x)[\chi_{J^\iota}(x)-\tilde{\chi}_{J^\iota}(x)]$ as
\Bes
\begin{split}
\int\mathcal{K}_j(x-y)b_Q(y)dy\int[\chi_{J^\iota}(x)-\chi_{J^\iota}(x-z)]
\varpi_{j-s\varkappa}(z)dz.
\end{split}
\Ees
Utilizing the Minkowski inequality and the Fubini theorem, we obtain
\Bes
\begin{split}
   &\quad\|(T_jb_{Q})[\chi_{J^\iota}-\tilde{\chi}_{J^\iota}]\|_{L^1(\mathbb{R}^d)}\\
     &\lc\int|b_Q(y)|\Big(\iint|\mathcal{K}_j(x-y)|\cdot|\chi_{J^\iota}(x)-\chi_{J^\iota}(x-z)|\cdot|\varpi_{j-s\varkappa}(z)|dzdx\Big)dy.
\end{split}
\Ees
Using the support of $\varpi_{j-s\varkappa}$, we get $|z|\leq2^{j-5-s\varkappa}$. By the support of $[\chi_{J^\iota}-\tilde{\chi}_{J^\iota}]$, only two cases happen: one is $x\in J^\iota$ and $x-z\in(J^\iota)^c$, the other is $x\in (J^\iota)^c$ and $x-z\in J^\iota$. Recall $J^{\iota}$ is a dyadic cube with sidelength $2^{j-1}$. Since $|z|<2^{j-5-s\varkappa}$, we can obtain that in both cases  $\dist\{x,\pari J^{\iota}\}\lc2^{j-s\varkappa}$ from a geometry observation. Hence we get
\Bes
\begin{split}
   \|(T_jb_{Q})[\chi_{J^\iota}-\tilde{\chi}_{J^\iota}]\|_{L^1(\mathbb{R}^d)}
     &\lc2^{-jd}\|\Om\|_\infty\int_{\dist\{x,\pari J^{\iota}\}\lc2^{j-s\varkappa}}dx\|b_Q\|_{L^1(\R^d)}\\
     &\lc2^{-s\varkappa}\|\Om\|_\infty\|b_Q\|_{L^1(\R^d)}
\end{split}
\Ees
which completes the proof of \eqref{e:27L1chi}.

On one hand, utilizing the triangle inequality and \eqref{e:27L1chi}, we get
\Be\label{e:27L1chi1}
\begin{split}
\Big\|\sum_{j\in\Z}\sum_{Q\in\mathfrak{Q}_{j-s}}&\epsilon_J (T_j b_{Q})(\chi_{J^\iota}-\tilde{\chi}_{J^\iota})\Big\|_{L^1(\mathbb{R}^d)} \\ &\lesssim\sum_{j\in\Z}\sum_{Q\in\mathfrak{Q}_{j-s}}\|(T_jb_{Q})[\chi_{J^\iota}-\tilde{\chi}_{J^\iota}]\|_{L^1(\mathbb{R}^d)}\\
&\lc2^{-s\varkappa}\|\Om\|_\infty\sum_{Q\in\mathfrak{Q}}\|b_Q\|_{L^1(\R^d)}.
\end{split}
\Ee

On the other hand, by Lemma \ref{l:27L^31}, we have
\Be\label{e:27L1chi2}
\begin{split}
\Big\|\sum_{j\in\Z}\sum_{Q\in\mathfrak{Q}_{j-s}}\epsilon_J (T_j b_{Q})(\chi_{J^\iota}-\tilde{\chi}_{J^\iota})\Big\|_{L^3(\mathbb{R}^d)} &\lc\Big\|\sum_{j\in\Z}\sum_{Q\in\mathfrak{Q}_{j-s}} |T_j b_{Q}|\Big\|_{L^3(\mathbb{R}^d)}\\
&\lesssim\|\Om\|_\infty\Big(\lambda^2\C_\Omega^{-2}\sum_{Q\in\mathfrak{Q}}\|b_Q\|_{L^1(\R^d)}\Big)^{1/3}.\end{split}
\Ee

Making an interpolation between \eqref{e:27L1chi1} and \eqref{e:27L1chi2}, we get
\Be\label{e:27L1chi3}
\begin{split}
\Big\|\sum_{j\in\Z}\sum_{Q\in\mathfrak{Q}_{j-s}}&\epsilon_J (T_j b_{Q})(\chi_{J^\iota}-\tilde{\chi}_{J^\iota})\Big\|_{L^2(\mathbb{R}^d)} \\ &\lc2^{-\fr14s\varkappa}\|\Om\|_\infty\Big(\lam\C_\Omega^{-1}\sum_{Q\in\mathfrak{Q}}\|b_Q\|_{L^1(\R^d)}\Big)^{\fr12}.
\end{split}
\Ee

Finally, combining \eqref{e:27L1chi3}, \eqref{e:27sumbQ} and $\|\Om\|_\infty\leq2^{s\eta}\|\Om\|_1$, we obtain
\Bes
\begin{split}
\Big\|\sum_{j\in\Z}\sum_{Q\in\mathfrak{Q}_{j-s}}&\epsilon_J (T_j b_{Q})(\chi_{J^\iota}-\tilde{\chi}_{J^\iota})\Big\|_{L^2(\mathbb{R}^d)} \\ &\lc2^{-\fr14\varkappa s+\eta s}2^{-n}\Big(\lam\C_\Omega\|f\|_{L^1(\R^d)}\Big)^{\fr12}
\end{split}
\Ees
 which completes the proof of Lemma \ref{l:27finall21} if we choose  constants $0<\eta<\fr14\varkappa$ and set $\del_1=2(\fr14\varkappa-\eta)$.
$\hfill{} \Box$
\vskip0.24cm

\subsection{Proof of Lemma \ref{l:27finall22}}
We first show the following $L^1$ estimate: For a fix $j\in\Z$ and $Q\in\mathfrak{Q}_{j-s}$,
\Be\label{e:27L1Pjs}
    \|P_{j-s\varkappa}[(T_j b_{Q})\tilde{\chi}_{J^\iota}]\|_{L^1(\mathbb{R}^d)} \lesssim 2^{-s(1-\varkappa)}\|\Omega\|_\infty\|b_Q\|_{L^1(\R^d)}.
\Ee

Utilizing the Fubini theorem, we write
\Bes
P_{j-s\varkappa}[(T_j b_{Q})\tilde{\chi}_{J^\iota}](x)
=\int b_Q(y)\Big[\int\varpi_{j-s\varkappa}(x-w)\mathcal{K}_j(w-y)\tilde{\chi}_{J^\iota}(w)dw\Big]dy.
\Ees
Let $y_0$ the center of $Q$. By making a change of variables $w-y=z$ and using the cancelation property of $b_Q$  and $\supp(b_Q)\subseteq Q$ (see (cz-ii)), the above integral equals to
\Bes
\int b_Q(y) \mathcal{K}_{j,s}^{x,\varkappa}(y) dy=\int_Q b_Q(y)[\mathcal{K}_{j,s}^{x,\varkappa}(y)-\mathcal{K}_{j,s}^{x,\varkappa}(y_0)]dy
\Ees
where
\Bes
\mathcal{K}_{j,s}^{x,\varkappa}(y)=\int\varpi_{j-s\varkappa}(x-y-z)\mathcal{K}_j(z)\tilde{\chi}_{J^\iota}(y+z)dz.
\Ees
By employing the mean value formula
$$
\mathcal{K}_{j,s}^{x,\varkappa}(y)-\mathcal{K}_{j,s}^{x,\varkappa}(y_0)=\int_0^1\inn{y-y_0}{\nabla [\mathcal{K}_{j,s}^{x,\varkappa}](ty+(1-t)y_0)}dt
$$
and $|y-y_0|\lc2^{j-s}$ for any $y\in Q$, we get
\Bes
\begin{split}
\|P_{j-s\varkappa}[(T_j b_{Q})\tilde{\chi}_{J^\iota}]\|_{L^1(\mathbb{R}^d)}
&\lc2^{j-s}\int_Q|b_Q(y)|\int|\mathcal{K}_j(z)|\Big[\|\nabla \varpi_{j-s\varkappa}\|_{L^1(\R^d)}\tilde{\chi}_{J^\iota}(y+z)\\
&\quad+\|\varpi_{j-s\varkappa}\|_{L^1(\R^d)}\int_0^1|\nabla \tilde{\chi}_{J^\iota}(ty+(1-t)y_0)|dt\Big]dzdy\\
&\lc2^{-(1-\varkappa)s}\|\Om\|_\infty\|b_Q\|_{L^1(\R^d)}
\end{split}
\Ees
where we also use $\|\nabla \varpi_{j-s\varkappa}\|_{L^1(\R^d)}\lc2^{-j+s\varkappa}$ and $\|\nabla\tilde{\chi}_{J^\iota}\|_{L^\infty(\R^d)}\lc2^{-j+s\varkappa}$ (see \eqref{e:27parix}). So we prove \eqref{e:27L1Pjs}.

On one hand, by the triangle inequality and \eqref{e:27L1Pjs}, we derive
\Be\label{e:27Ipjs1}
\begin{split}
\Big\|\sum_{j\in\Z}\sum_{Q\in\mathfrak{Q}_{j-s}}P_{j-s\varkappa}[ \epsilon_J(T_j b_{Q})\tilde{\chi}_{J^\iota}]\Big\|_{L^1(\mathbb{R}^d)}
&\leq \sum_{j\in\Z}\sum_{Q\in\mathfrak{Q}_{j-s}}\Big\|P_{j-s\varkappa}[ (T_j b_{Q})\tilde{\chi}_{J^\iota}]\Big\|_{L^1(\mathbb{R}^d)}\\
&\lc2^{-(1-\varkappa)s}\|\Om\|_\infty\sum_{Q\in\mathfrak{Q}}\|b_Q\|_{L^1(\R^d)}.
\end{split}
\Ee

On the other hand, by the triangle inequality and Lemma \ref{l:27L^32}, we obtain
\Be\label{e:27Ipjs2}
\begin{split}
\Big\|\sum_{j\in\Z}\sum_{Q\in\mathfrak{Q}_{j-s}}P_{j-s\varkappa}[\epsilon_J (T_j b_{Q})\tilde{\chi}_{J^\iota}]\Big\|_{L^3(\mathbb{R}^d)}&
\leq \Big\|\sum_{j\in\Z}\sum_{Q\in\mathfrak{Q}_{j-s}}P_{j-s\varkappa}[ |T_j b_{Q}|]\Big\|_{L^3(\mathbb{R}^d)}\\
&\lc\|\Om\|_\infty\Big(\lambda^2\C_\Omega^{-2}\sum_{Q\in\mathfrak{Q}}\|b_Q\|_{L^1(\R^d)}\Big)^{1/3}.
\end{split}
\Ee

Making an interpolation between \eqref{e:27Ipjs1} and \eqref{e:27Ipjs2}, together with \eqref{e:27sumbQ} and $\|\Om\|_\infty\leq 2^{\eta s}\|\Om\|_1$, we get
\Bes
    \Big\|\sum_{j\in\Z}\sum_{Q\in\mathfrak{Q}_{j-s}}P_{j-s\varkappa}[\epsilon_J (T_j b_{Q})\tilde{\chi}_{J^\iota}]\Big\|_{L^2(\mathbb{R}^d)} \lesssim 2^{-\fr14(1-\varkappa){s}+\eta s}2^{-n}\big(\lam\C_\Om\|f\|_{L^1(\mathbb{R}^d)}\big)^{\fr12}
\Ees
which completes the proof of Lemma \ref{l:27finall22} if we choose $0<\eta<\fr14(1-\varkappa)$ and set $\del_2=\fr12(1-\varkappa)-2\eta$.
$\hfill{} \Box$
\vskip0.24cm

\subsection{Proof of Lemma \ref{l:27finall23}}

We straightforwardly use $TT^*$ argument to deal with this lemma. We begin by stating an orthogonality property concerning the support of $\mathcal{F}(G^s_{v})$: For a fixed $s\geq 200$,
\begin{equation}\label{e:27obser}
\sup\limits_{\xi\neq0}\sum\limits_{e^s_v\in\Theta_{s}}|\Phi^2(2^{s\ga}\inn{e^s_v}{\xi/|\xi|})|\lc2^{s\ga(d-2)}.
\end{equation}
To verify this, note that by the homogeneity of  $\Phi^2(2^{s\ga}\inn{e^s_v}{\xi/|\xi|})$, it is sufficient to take the supremum over the surface $S^{d-1}$. For $|\xi|=1$ and $\xi\in\supp[\Phi^2(2^{s\ga}\inn{e^s_v}{\xi/|\xi|})]$, let $\xi^{\bot}$ denote the hyperplane perpendicular to $\xi$. Then it is easy to see that
\begin{equation}\label{e:27e^n_v}
\text{dist}(e^s_v,\xi^\bot)\lc2^{-s\ga}.
\end{equation}
Given that the mutual distance between the vectors $e^s_v$'s is $C2^{-s\ga}$, the number of vectors satisfying \eqref{e:27e^n_v} is at most $2^{s\ga(d-2)}$. This establishes \eqref{e:27obser}.

Applying the Plancherel theorem, the Cauchy-Schwarz inequality and finally the Plancherel theorem again, we obtain
\begin{equation}\label{e:27L^2key}
\begin{split}
&\Big\|\sum\limits_{v}G^s_{v}\sum\limits_j\sum_{Q\in\mathfrak{Q}_{j-s}}(I-P_{j-s\varkappa})[\epsilon_J(T^{v}_jb_{Q})\tilde{\chi}_{J^\iota}]\Big\|^2_{L^2(\R^d)}\\
&=\int\Big|\sum\limits_{v}\Phi(2^{s\ga}\inn{e^s_v}{\xi/|\xi|})\mathcal{F}\Big(\sum\limits_j\sum_{Q\in\mathfrak{Q}_{j-s}}(I-P_{j-s\varkappa})[\epsilon_J(T^{v}_jb_{Q})\tilde{\chi}_{J^\iota}]\Big)(\xi)\Big|^2d\xi\\
&\lc2^{s\ga(d-2)} \Big\|\sum\limits_{v}\Big|\mathcal{F}\Big(\sum\limits_j\sum_{Q\in\mathfrak{Q}_{j-s}}(I-P_{j-s\varkappa})[\epsilon_J(T^{v}_jb_{Q})\tilde{\chi}_{J^\iota}]\Big)\Big|^2\Big\|_{L^1(\R^d)}\\
&\lc2^{s\ga(d-2)} \sum\limits_{v}\Big\|\sum\limits_j\sum_{Q\in\mathfrak{Q}_{j-s}}(I-P_{j-s\varkappa})[\epsilon_J(T^{v}_jb_{Q})\tilde{\chi}_{J^\iota}]\Big\|^2_{L^2(\R^d)}.\\
\end{split}
\end{equation}

In the following, we claim that for a fixed $e^s_v$,
\begin{equation}\label{e:27L^2}
\Big\|\sum\limits_j\sum_{Q\in\mathfrak{Q}_{j-s}}(I-P_{j-s\varkappa})[\epsilon_J(T^{v}_jb_{Q})\tilde{\chi}_{J^\iota}]\Big\|^2_{L^2(\R^d)}\\
\lc2^{-2s\ga(d-1)+2s\eta-2n}\lam\C_\Om\|f\|_{L^1(\R^d)}.
\end{equation}
Using this estimate, $\card(\Theta_s)\lc2^{s\ga(d-1)}$ and (\ref{e:27L^2key}), we get
\begin{equation*}
\Big\|\sum\limits_{v}\sum\limits_j\sum_{Q\in\mathfrak{Q}_{j-s}}(I-P_{j-s\varkappa})G^s_{v}[\epsilon_J(T^{v}_jb_{Q})\tilde{\chi}_{J^\iota}]\Big\|^2_{L^2(\R^d)}\\
\lc2^{-s\ga+2s\eta-2n}\lam\C_\Omega\|f\|_{L^1(\R^d)},
\end{equation*}
which is just the desired bound of Lemma \ref{l:27finall23} if we choose $0<\eta<\fr12\ga$ and set $\del_3=2(\ga-2\eta)$. Thus, to finish the proof of Lemma \ref{l:27finall23}, it is enough to prove  (\ref{e:27L^2}).

Recall \eqref{e:27IPjs}, we have shown that 
\Bes
|(I-P_{j-s\varkappa})[\epsilon_J(T^v_j b_{Q})\tilde{\chi}_{J^\iota}](x)|\lc H_j^{s,v}*|b_Q|(x)
\Ees
where $H_j^{s,v}(x):=2^{-jd}\chi_{E^{s,v}_j}(x)\|\Om\|_\infty$.
Recall we also define $\mathfrak{B}_s=\sum_{Q\in\mathfrak{Q}_s}|b_Q|$. Then for a fixed $e^s_v$, we obtain
\begin{equation}\label{e:27L^2key2}
\begin{split}
\Big\|\sum\limits_j\sum_{Q\in\mathfrak{Q}_{j-s}}&(I-P_{j-s\varkappa})[\epsilon_J(T^{v}_jb_{Q})\tilde{\chi}_{J^\iota}]\Big\|^2_{L^2(\R^d)}
\lc\int_{\R^d}\Big|\sum\limits_jH^{s,v}_j*\mathfrak{B}_{j-s}(x)\Big|^2dx\\
&\leq2\sum\limits_j\sum_{i\leq j}\int_{\R^d} H_i^{n,v}*\mathfrak{B}_{i-s}(x) \cdot H_j^{s,v}*\mathfrak{B}_{j-s}(x) dx\\
&=2\sum\limits_j\sum\limits_{i\leq j}\int_{\R^d} H_j^{s,v}*H_i^{s,v}*\mathfrak{B}_{i-s}(x)\cdot\mathfrak{B}_{j-s}(x)dx.
\end{split}
\end{equation}
Observe that $\|H_i^{s,v}\|_{L^1(\R^d)}\lc2^{-id}\cdot|E_i^{s,v}|\cdot\|\Omega\|_\infty\lc2^{-s\ga(d-1)}\|\Omega\|_\infty$, therefore we get that for any $i\leq j$,
$$H_j^{s,v}*H_i^{s,v}(x)\lc 2^{-s\ga(d-1)}2^{-jd}\|\Omega\|_\infty^2\chi_{\widetilde{E}^{s,v}_j}(x),$$
where $\widetilde{E}^{s,v}_j=E^{s,v}_j+E^{s,v}_j$.
Hence for a fixed $j$, $e^s_v$ and $x$, we derive
\begin{equation}\label{e:27L^2con}
\begin{split}
\sum\limits_{i\leq j}H_j^{s,v}*H_i^{s,v}*\mathfrak{B}_{i-s}(x)&\lc2^{-s\ga(d-1)}2^{-jd}\|\Omega\|^2_\infty\sum\limits_{i\leq j}\int_{x+\widetilde{E}^{s,v}_j}\mathfrak{B}_{i-s}(y)dy\\
&\lc2^{-s\ga(d-1)}2^{-jd}\|\Omega\|^2_\infty\sum\limits_{i\leq j}\sum\limits_{Q\in\mathfrak{Q}_
{i-s}\atop Q\cap\{x+\widetilde{E}^{s,v}_j\}\neq\emptyset}\int_{Q}|b_Q(y)|dy.
\end{split}
\end{equation}
Now applying $\|\Om\|_\infty\leq2^{\eta s}\|\Om\|_1$ and the Calder\'on-Zygmund decomposition property (cz-ii): $\int|b_Q(y)|dy\lc\lam|Q|/{\C_\Om}$, the above estimate is bounded by
\begin{equation}\label{e:27L^2con2}
\begin{split}
2^{-s\ga(d-1)+2s\eta}&2^{-jd}\C_\Om^2\sum\limits_{i\leq j}\sum\limits_{Q\in\mathfrak{Q}_
{i-s}\atop Q\cap\{x+\widetilde{E}^{s,v}_j\}\neq\emptyset}{\lam}{\C_\Om^{-1}}|Q|\\
&\lc2^{-s\ga(d-1)+2s\eta}2^{-jd}2^{jd-s\ga(d-1)}{\lam}{\C_\Om}={\lam}{\C_\Om} 2^{-2s\ga(d-1)+2s\eta},
\end{split}
\end{equation}
where we also use fact that all the cubes in $\mathfrak{Q}$ are disjoint (see (cz-iv)). By (\ref{e:27L^2key2}),
 (\ref{e:27L^2con}), \eqref{e:27L^2con2} and \eqref{e:27sumbQ}, we obtain
\Bes
\begin{split}
\Big\|\sum\limits_j\sum_{Q\in\mathfrak{Q}_{j-s}}(I-P_{j-s\varkappa})[\epsilon_J(T^{v}_jb_{Q})\tilde{\chi}_{J^\iota}]\Big\|^2_{L^2(\R^d)}
&\lc\lam2^{-2s\ga(d-1)+2s\eta}\C_\Omega\sum\limits_{Q\in\mathfrak{Q}}\|b_{Q}\|_{L^1(\R^d)}\\
&\lc\lam2^{-2s\ga(d-1)+2s\eta-2n}\C_\Omega\|f\|_{L^1(\R^d)},
\end{split}
\Ees
which is the asserted bound for \eqref{e:27L^2}. Hence, we complete the proof of Lemma \ref{l:27finall23}.
$\hfill{} \Box$
\vskip0.24cm

\section{Proof of Lemma \ref{l:27finall24}}\label{s:275}

Our strategy to prove Lemma \ref{l:27finall24} is similar to that of Lemma \ref{l:27finall21} and \ref{l:27finall22}, i.e. making an interpolation between an $L^1$ estimate and an $L^3$ estimate.
To establish the $L^1$ estimate with an exponential decay in $s$, we will apply the stationary phase method to deal with some oscillatory integrals.

Let us introduce some notation. We begin by defining the Littlewood-Paley decomposition. Let $\psi$ be a $C^\infty$ function supported in $\{\xi : \frac{1}{2} \leq |\xi| \leq 2\}$ such that
$$\sum_{k\in\Z}\psi^2(2^{k}\xi)=1\ \text{ for all }\ \ \xi\in\R^d\setminus\{0\}.$$  Define $\psi_k(\xi) = \psi(2^k \xi)$. Choose $\tilde{\psi}$ to be a radial $C^\infty$ function such that $\tilde{\psi}(\xi) = 1$ for $\frac{1}{2} \leq |\xi| \leq 2$, $\operatorname{supp}(\tilde{\psi}) \subseteq \{\xi : \frac{1}{4} \leq |\xi| \leq 4\}$, and $0 \leq \tilde{\psi}(\xi) \leq 1$ for all $\xi \in \mathbb{R}^d$. Define $\tilde{\psi}_k(\xi) = \tilde{\psi}(2^k \xi)$. Then it is clear that $\psi_k = \tilde{\psi}_k \psi_k$.
Define the convolution operators $\Lambda_k$ and $\tilde{\Lambda}_k$ via Fourier multipliers $\psi_k$ and $\tilde{\psi}_k$, respectively:
$$\widehat{{\Lam}_kg}(\xi)=\psi_k(\xi)\hat{g}(\xi),\quad \ \widehat{\tilde{\Lam}_kg}(\xi)=\tilde{\psi}_k(\xi)\hat{g}(\xi).$$
By construction, we have $\Lambda_k = \tilde{\Lambda}_k \Lambda_k$ and the identity operator satisfies
$I=\sum\limits_{k\in\mathbb{Z}}\Lam_k^2.$
Hence for a fixed $j\in\Z$, $Q\in\mathfrak{Q}_{j-s}$ and $e^s_v\in\Theta_s$, we derive that
$$(I-P_{j-s\varkappa})(I-G^s_v)[\epsilon_J(T^v_j b_{Q})\tilde{\chi}_{J^\iota}]=\sum_k(I-P_{j-s\varkappa})(I-G^s_v)\Lam_k^2[\epsilon_J(T^v_j b_{Q})\tilde{\chi}_{J^\iota}].$$
Using the triangle inequality, we get
\Be\label{e:27sumk}
\begin{split}
&\quad\|(I-P_{j-s\varkappa})(I-G^s_v)[\epsilon_J(T^v_j b_{Q})\tilde{\chi}_{J^\iota}]\|_{L^1(\R^d)}\\
&\leq\sum_k\|(I-P_{j-s\varkappa})\Lam_k(I-G^s_v)\Lam_k[\epsilon_J(T^v_j b_{Q})\tilde{\chi}_{J^\iota}]\|_{L^1(\R^d)}\\
&\leq\sum_k\|(I-P_{j-s\varkappa})\Lam_k\|_{L^1(\R^d)\rta L^1(\R^d)}\|(I-G^s_v)\Lam_k[\epsilon_J(T^v_j b_{Q})\tilde{\chi}_{J^\iota}]\|_{L^1(\R^d)}.
\end{split}
\Ee

For the term $(I-P_{j-s\varkappa})\Lam_k$, we have the following estimate.
\begin{lemma}\label{l:27IPjslamk}
With all notions above, 
\Bes
\|(I-P_{j-s\varkappa})\Lam_k\|_{L^1(\R^d)\rta L^1(\R^d)}\lc\min\{1,2^{j-s\varkappa-k}\}.
\Ees
\end{lemma}
\begin{proof}
On one hand,  it is easy to see that 
  \Bes
  \|(I-P_{j-s\varkappa})\Lam_kg\|_{L^1(\R^d)}\leq\|\Lam_kg\|_{L^1(\R^d)}+\|P_{j-s\varkappa}\Lam_kg\|_{L^1(\R^d)}\lc\|g\|_{L^1(\R^d)}.
  \Ees

On the other hand, we could write $(I-P_{j-s\varkappa})\Lam_kg(x)$ as
\Bes
\begin{split}[\mathcal{F}^{-1}(\psi_k)&-\varpi_{j-s\varkappa}*\mathcal{F}^{-1}(\psi_k)]*g(x)\\
&=\int\Big[\int\big(\mathcal{F}^{-1}(\psi_k)(x-y)-\mathcal{F}^{-1}(\psi_k)(x-y-z)\big)\varpi_{j-s\varkappa}(z)dz\Big]g(y)dy.
\end{split}
\Ees
By the support of $\varpi_{j-s\varkappa}$, we get $|z|\lc 2^{j-s\varkappa}$. Utilizing the mean value formula and the Minkowski inequality, we then obtain
  \Bes
  \begin{split}
  \|(I-P_{j-s\varkappa})\Lam_kg\|_{L^1(\R^d)}&\lc2^{j-s\varkappa}\|\nabla[\mathcal{F}^{-1}(\psi_k)]\|_{L^1(\R^d)}\|\varpi_{j-s\varkappa}\|_{L^1(\R^d)}\|g\|_{L^1(\R^d)}\\
  &\lc2^{j-s\varkappa-k}\|g\|_{L^1(\R^d)}.
  \end{split}\Ees
Now combining the above two estimates, we finish the proof.
\end{proof}
For the term $(I-G^s_v)\Lam_k[\epsilon_J(T^v_j b_{Q})\tilde{\chi}_{J^\iota}]$, we have the following two distinct $L^1$ estimates for its high-frequency and low-frequency parts.

\begin{lemma}\label{l:27l^1_1}
For a fixed $j\in\Z$, $Q\in\mathfrak{Q}_{j-s}$, $e^s_v\in\Theta_s$ and $k\in\Z$, there exists $N>0$ such that for any $N_1\in\Z_+$
\Be\label{e:27main}
\begin{split}
\|(I-G^s_v)\Lam_k[&\epsilon_J(T^v_j b_{Q})\tilde{\chi}_{J^\iota}]\|_{L^1(\R^d)}\\
&\lc C_{N_1}2^{-s\ga(d-1)-(j-k)N_1+2s\ga N+s\varkappa N_1+s\gamma N_1+s\eta}\C_\Omega\|b_Q\|_{L^1(\R^d)}.
\end{split}\Ee
\end{lemma}

\begin{proof}
Utilizing the Fubini theorem and $\supp(b_Q)\subseteq Q$, we write the function $(I-G^s_v)\Lam_k[\epsilon_J(T^v_j b_{Q})\tilde{\chi}_{J^\iota}](x)$ as
\begin{equation*}
\begin{split}
\quad(I-G^s_v)\Lam_k[\epsilon_J(T^v_j b_{Q})\tilde{\chi}_{J^\iota}](x)\triangleq\epsilon_J\int_{Q}b_Q(y)D_{j,k}^{s,v}(x,y)dy
\end{split}
\end{equation*}
where $D_{j,k}^{s,v}(x,y)$ is defined as the kernel of $(I-G^s_v)\Lam_k[(T^v_j b_{Q})\tilde{\chi}_{J^\iota}](x)$. More precisely, $D_{j,k}^{s,v}(x,y)$ equals to
$$\fr{1}{(2\pi)^d}\int_{\R^d} e^{ix\cdot\xi}H_{k,s,v}(\xi)\int_{\R^d} e^{-i\xi\cdot\om}\Om(\om-y)\chi_{E_v^s}{\Big(\fr{\om-y}{|\om-y|}\Big)}\fr{\varphi_j(\om-y)}{|\om-y|^d}\tilde{\chi}_{J^\iota}(\om) d\om d\xi
$$
where $H_{k,s,v}(\xi)=(1-\Phi(2^{s\ga}\inn{e^s_v}{\xi/|\xi|}))\psi_{k}(\xi).$
Using the Minkowski inequality, we get
\Bes
\|(I-G^s_v)\Lam_k[\epsilon_J(T^v_j b_{Q})\tilde{\chi}_{J^\iota}]\|_{L^1(\R^d)}
\leq\sup_{y\in Q}\|D_{j,k}^{s,v}(\cdot,y)\|_{L^1(\R^d)}\|b_{Q}\|_{L^1(\R^d)}.
\Ees

Hence in the following we only need to give an $L^1$ estimate of $D_{j,k}^{s,v}(\cdot,y)$ for a fixed $y\in Q$.
In order to separate the rough kernel, we make a change of variables $\om-y=r\theta$.
By the Fubini theorem, the kernel $D_{j,k}^{s,v}(x,y)$ can be written as
\begin{equation}\label{e:27mainintegral}
\fr{1}{(2\pi)^{d}}\int_{E^{s}_v}\Om(\theta)
\bigg[\int_{\R^d}\int_0^\infty
e^{i\inn{x-y-r\theta}{\xi}}H_{k,s,v}(\xi)\fr{{\varphi}_j(r)}{r}\tilde{\chi}_{J^\iota}(y+r\tet) drd\xi\bigg]d\si(\tet).
\end{equation}
Concerning the support of $\varphi_j(r)$, we have $2^{j-4}\leq r\leq2^{j-2}$. Integrating by parts $N_1$ times with $r$, the integral involving $r$ then can
be rewritten as
$$\int_0^\infty e^{-i\inn{r\theta}{\xi}}(i\inn{\theta}{\xi})^{-N_1}
\pari^{N_1}_r\Big[\fr{{\varphi}_j(r)}{r}\tilde{\chi}_{J^\iota}(y+r\tet)\Big]dr.$$

Since $\theta\in E^s_v$, then $|\theta-e^s_v|\leq 2^{-s\ga-2}$. By the support of $\Phi$,
we see $|\inn{e^s_v}{\xi/|\xi|}|\geq 2^{1-s\ga}$. So we obtain
\begin{equation}\label{e:27ang}
|\inn{\theta}{\xi/|\xi|}|\geq|\inn{e^s_v}{\xi/|\xi|}|-|\inn{e^s_v-\theta}{\xi/|\xi|}|\geq2^{-s\ga}.
\end{equation}

Next integrating by parts with $\xi$, the integral in \eqref{e:27mainintegral}
can be rewritten as
\begin{equation}\label{e:27minte}
\begin{split}
\fr{1}{(2\pi)^d}\int_{E^{s}_v}&\Om(\tet)\int_{\R^d}e^{i\inn{x-y-r\tet}{\xi}}\int_0^\infty
\pari_r^{N_1}\Big(\fr{{\varphi}_j(r)}{r}\tilde{\chi}_{J^\iota}(y+r\tet)\Big)\times\\
&\fr{(I-2^{-2k}\Delta_\xi)^N}{(1+2^{-2k}|x-y-r\tet|^2)^N}\Big(H_{k,s,v}(\xi)(i\inn{\tet}{\xi})^{-N_1}\Big)drd\xi d\si(\tet).
\end{split}
\end{equation}

In the following, we give an explicit estimate of the term in (\ref{e:27minte}).
Utilizing the product rule,
\begin{equation}\label{e:27parir}
\begin{split}
\Big|\pari^{N_1}_r\Big(\fr{{\varphi}_j(r)}{r}\tilde{\chi}_{J^\iota}(y+r\tet)\Big)\Big|&=\Big|\sum_{i=0}^{N_1}C_{N_1}^i\pari_r^i[\tilde{\chi}_{J^\iota}(y+r\tet)]\pari_r^{N_1-i}\Big[\fr{{\varphi}_j(r)}{r}\Big]\Big|.\\
\end{split}
\end{equation}
Applying \eqref{e:27parix} and $2^{j-4}\leq r\leq2^{j-2}$ , the above term \eqref{e:27parir} is majorized by
\Be\label{e:27rbound}
\sum\limits_{i=0}^{N_1}C_{N_1}^i2^{-(j-s\varkappa)i}2^{j(-1-N_1+i)}\\
\leq C_{N_1}2^{\varkappa N_1}2^{-(1+N_1)j}.
\Ee

Below we demonstrate the inequality
\begin{equation}\label{e:27D2}
\big|(I-2^{-2k}\Delta_\xi)^{N}[\inn{\theta}{\xi}^{-N_1}H_{k,s,v}(\xi)]\big|\leq C_{N_1}2^{(s\ga+k)N_1+2s\ga N}.
\end{equation}
We begin by proving \eqref{e:27D2} for the case
$N=0$. From \eqref{e:27ang}, it follows that
$$|(-i\inn{\theta}{\xi})^{-N_1}\cdot H_{k,s,v}(\xi)|\lc|\inn{\theta}{\xi}|^{-N_1}\lc C_{N_1}2^{(s\ga+k)N_1}.$$
Utilizing the product rule, we compute
\begin{equation*}
|\pari_{\xi_i}H_{k,n,v}(\xi)|=\big|-\pari_{\xi_i}[\Phi(2^{s\ga}\inn{e^s_v}{\xi/|\xi|})]
\cdot\psi_{k}(\xi)+\pari_{\xi_i}\psi_{k}(\xi)\cdot
(1-\Phi(2^{s\ga}\inn{e^s_v}{\xi/|\xi|}))\big|
\lc2^{s\ga+k}.
\end{equation*}
By induction, it follows that for any multi-indices $\alpha\in\Z^d_+$,
$|\pari_{\xi}^{\alpha}H_{k,s,v}(\xi)|\lc2^{(s\ga+k)|\alpha|}$.
Applying the product rule again together with \eqref{e:27ang}, we derive
\begin{equation*}
\begin{split}
&\quad\Big|\pari^2_{\xi_i}\big[(\inn{\theta}{\xi})^{-N_1}H_{k,s,v}(\xi)\big]\Big|\\
&=\big|{\inn{\theta}{\xi}}^{-N_1-2}\cdot N_1(N_1+1)\theta_i^2\cdot H_{k,s,v}\\
&\quad+2{\inn{\theta}{\xi}}^{-N_1-1}\cdot(-N_1)\cdot\theta_i\pari_{\xi_i}H_{k,s,v}(\xi)
+\inn{\theta}{\xi}^{-N_1}\pari_{\xi_i}^2 H_{k,s,v}(\xi)\big|\\
&\leq C_{N_1}2^{(s\ga+k)(N_1+2)}.\\
\end{split}
\end{equation*}
Hence we conclude that
\begin{equation*}
2^{-2k}\big|\Delta_\xi[(\inn{\theta}{\xi})^{-N_1}H_{k,s,v}(\xi)]\big|\leq C_{N_1}2^{(s\ga+k)N_1+2s\ga}.
\end{equation*}
The general case of \eqref{e:27D2} follows by induction on $N$.

Now we choose $N=[d/2]+1$.
To obtain the $L^1$ estimate of (\ref{e:27mainintegral}), we note that by the support of $H_{k,s,v}$,
$$\int_{\supp(H_{k,s,v})}\int\Big(1+2^{-2k}|x-y-r\theta|^2\Big)^{-N}dxd\xi\lc 1.$$

Integrating in $r$ yields a bound $2^j$. Recalling the assumption that $\|\Om\|_\infty\leq 2^{s\eta}\|\Omega\|_1$. Then integrating in $\theta$, we get a bound $2^{-s\ga(d-1)+s\eta}\C_\Omega$. Combining the estimates \eqref{e:27parir}, \eqref{e:27rbound}, \eqref{e:27D2} with the bounds above, we obtain
\Bes
\begin{split}
\|D_{j,k}^{s,v}(\cdot,y)\|_{L^1(\R^d)}&\lc C_{N_1}2^{s{\varkappa}N_1-j(1+N_1)+(s\ga+k)N_1+2s\ga N+j-s\ga(d-1)+s\eta}\C_\Omega\\
&=C_{N_1}2^{-s\ga(d-1)-(j-k)N_1+2s\ga N+s\varkappa N_1+s\gamma N_1+s\eta}\C_\Omega\\
\end{split}
\Ees
holds for any $y\in Q$. This consequently implies the desired bound for Lemma \ref{l:27l^1_1} with $N=[\fr{d}{2}]+1$.
\end{proof}

\vskip0.24cm

\begin{lemma}\label{l:27l^1_2}
For a fixed $j\in\Z$, $Q\in\mathfrak{Q}_{j-s}$, $e^s_v\in\Theta_s$ and $k\in\Z$, we get
\Bes
\|(I-G^s_v)\Lam_k[\epsilon_J(T^v_j b_{Q})\tilde{\chi}_{J^\iota}]\|_{L^1(\R^d)}\lc2^{-s\ga(d-1)+s\eta}(2^{-(1-\varkappa)s}+2^{j-s-k})\C_\Om\|b_Q\|_{L^1(\R^d)}.
\Ees
\end{lemma}
\begin{proof}
Using $\Lam_k=\Lam_k\tilde{\Lam}_k$, we get
\Bes
\|(I-G^s_v)\Lam_k[\epsilon_J(T^v_j b_{Q})\tilde{\chi}_{J^\iota}]\|_{L^1(\R^d)}\leq \|(I-G^s_v)\tilde{\Lam}_k\|_{L^1(\R^d)\rta L^1(\R^d)}\|\Lam_k[(T^v_j b_{Q})\tilde{\chi}_{J^\iota}]\|_{L^1(\R^d)}.
\Ees
It is straightforward to check that
\Bes
\|(I-G^s_v)\tilde{\Lam}_k\|_{L^1(\R^d)\rta L^1(\R^d)}\lc1
\Ees
holds  uniformly for $s,e^s_v,k$ (see for example \cite[page 100]{See96}).
Thus, to finish the proof of Lemma \ref{l:27l^1_2}, it suffices to show the following estimate
\Be\label{e:27lamkTvj}
\|\Lam_k[(T^v_j b_{Q})\tilde{\chi}_{J^\iota}]\|_{L^1(\R^d)}
\lc2^{-s\ga(d-1)+s\eta}(2^{-(1-\varkappa)s}+2^{j-s-k})\C_\Om\|b_Q\|_{L^1(\R^d)}.
\Ee

Let $y_0$ is the center of $Q$. By the Fubini theorem, the cancellation property of $b_Q$ and $\supp(b_Q)\subseteq Q$ (see (cz-ii)), we can write
\Bes
\Lam_k[(T^v_j b_{Q})\tilde{\chi}_{J^\iota}](x)\triangleq\int_Q A_{j,k}^{s,v}(x,y) b_Q(y)dy=\int_Q [A_{j,k}^{s,v}(x,y)-A_{j,k}^{s,v}(x,y_0)]b_Q(y)dy
\Ees
where the kernel $A_{j,k}^{s,v}(x,y)$ is defined as
$$
\fr{1}{(2\pi)^d}\int_{\R^d} e^{ix\cdot\xi}\psi_k(\xi)\int_{\R^d} e^{-i\xi\cdot\om}\Om(\om-y)\chi_{E_v^s}{\Big(\fr{\om-y}{|\om-y|}\Big)}\fr{\varphi_j(\om-y)}{|\om-y|^d}\tilde{\chi}_{J^\iota}(\om) d\om d\xi.
$$

By making a change of variables to polar coordinates $w-y=r\tet$ and applying the Fubini theorem, we can write $A_{j,k}^{s,v}(x,y)$ as
\begin{equation}\label{e:27kjintegral}
\fr{1}{(2\pi)^{d}}\int_{E^{s}_v}\Om(\theta)
\bigg[\int_{\R^d}\int_0^\infty
e^{i\inn{x-y-r\theta}{\xi}}\psi_{k}(\xi)\fr{{\varphi}_j(r)}{r}\tilde{\chi}_{J^\iota}(y+r\tet) drd\xi\bigg]d\si(\tet).
\end{equation}
Integrating by part $N=[d/2]+1$ times with $\xi$, the above integral then equals to
\begin{equation}\label{e:27kjomintegral}
\begin{split}\fr{1}{(2\pi)^{d}}\int_{E^{s}_v}\Om(\theta)
\bigg[\int_{\R^d}\int_0^\infty
&e^{i\inn{x-y-r\theta}{\xi}}\fr{{\varphi}_j(r)}{r}\tilde{\chi}_{J^\iota}(y+r\tet) \\
&\quad\times\fr{(I-2^{-2k}\Delta_\xi)^{N}\psi_k(\xi)}
{\big(1+2^{-2k}|x-y-r\theta|^2\big)^{N}}drd\xi\bigg]d\si(\tet).
\end{split}
\end{equation}

Next we write
\Bes
A_{j,k}^{s,v}(x,y)-A_{j,k}^{s,v}(x,y_0)\triangleq A_{j,k,1}^{s,v}(x,y)+A_{j,k,2}^{s,v}(x,y)+A_{j,k,3}^{s,v}(x,y)
\Ees
where
\begin{equation*}
\begin{split}
A_{j,k,1}^{s,v}(x,y)=\fr{1}{(2\pi)^{d}}\int_{E^{s}_v}&\Om(\theta)
\bigg[\int_{\R^d}\int_0^\infty
e^{i\inn{x-r\theta}{\xi}}\Big(e^{-i\inn{y}{\xi}}-e^{-i\inn{y_0}{\xi}}\Big)\fr{{\varphi}_j(r)}{r}
\\
&\quad\times\tilde{\chi}_{J^\iota}(y+r\tet) \fr{(I-2^{-2k}\Delta_\xi)^{N}\psi_k(\xi)}
{\big(1+2^{-2k}|x-y-r\theta|^2\big)^{N}}drd\xi\bigg]d\si(\tet),
\end{split}
\end{equation*}
\begin{equation*}
\begin{split}
A_{j,k,2}^{s,v}(x,y)=\fr{1}{(2\pi)^{d}}\int_{E^{s}_v}\Om(\theta)&
\bigg[\int_{\R^d}\int_0^\infty
e^{i\inn{x-y_0-r\theta}{\xi}}
\Big(\tilde{\chi}_{J^\iota}(y+r\tet)-\tilde{\chi}_{J^\iota}(y_0+r\tet)\Big)\\
&\quad\times \fr{{\varphi}_j(r)}{r}\fr{(I-2^{-2k}\Delta_\xi)^{N}\psi_k(\xi)}
{\big(1+2^{-2k}|x-y-r\theta|^2\big)^{N}}drd\xi\bigg]d\si(\tet),
\end{split}
\end{equation*}
and
\begin{equation*}
\begin{split}
A_{j,k,3}^{s,v}(x,y)&=\fr{1}{(2\pi)^{d}}\int_{E^{s}_v}\Om(\theta)
\bigg[\int_{\R^d}\int_0^\infty
e^{i\inn{x-y_0-r\theta}{\xi}}\fr{{\varphi}_j(r)}{r}
\tilde{\chi}_{J^\iota}(y_0+r\tet)\\
&\times \Big(\fr{{(I-2^{-2k}\Delta_\xi)^{N}\psi_k(\xi)}}{\big(1+2^{-2k}|x-y-r\theta|^2\big)^{N}}-\fr{{(I-2^{-2k}\Delta_\xi)^{N}\psi_k(\xi)}}{\big(1+2^{-2k}|x-y_0-r\theta|^2\big)^{N}}\Big)drd\xi\bigg]d\si(\tet).
\end{split}
\end{equation*}
Hence by the Minkowski inequality, we get
\begin{equation}\label{e:27A_m}
\|\Lam_k[(T^v_j b_{Q})\tilde{\chi}_{J^\iota}]\|_{L^1(\R^d)}\leq \sup_{y\in Q}\sum_{\ell=1}^3\|A_{j,k,\ell}^{s,v}(\cdot,y)\|_{L^1(\R^d)}\|b_Q\|_{L^1(\R^d)}.
\end{equation}

\medskip

\textbf{Estimate of $A_{j,k,1}^{s,v}$}.
We employ a method similar to the proof of Lemma \ref{l:27l^1_1} but without applying integrating by parts. Note that $y\in Q$ and $y_0$
is the center of $Q$, then $|y-y_0|\lc 2^{j-s}$. This implies
$$\Big|e^{i\inn{-y}{\xi}}-e^{i\inn{-y_0}{\xi}}\Big|\lc 2^{j-s-k}.$$
Since $2^{j-4}\leq r\leq2^{j-2}$, we have
$|\varphi_j(r)r^{-1}|\lc 2^{-j}.$
Furthermore, we observe that
$$|(I-2^{-2k}\Delta_\xi)^N\psi_k(\xi)|\lc 1.$$
To estimate the $L^1$ estimate of $A_{j,k,1}^{s,v}(\cdot,y)$, we note that by the support of $\psi_k(\xi)$
$$\int_{|\xi|\lc2^{-k}}\int\Big(1+2^{-2k}|x-y-r\theta|^2\Big)^{-N}dxd\xi\lc 1.$$
Integrating in $r$, we get a bound $2^j$. Recall our assumption $\|\Om\|_\infty\leq 2^{s\eta}\|\Om\|_1$, so integrating in $\theta$ gives a bound $2^{-s\ga(d-1)+s\eta}\C_\Omega$.
Combining these bounds, we obtain that
\Be\label{e:27Ajk1}
\|A_{j,k,1}^{s,v}(\cdot,y)\|_{L^1(\R^d)}\lc2^{-s\ga(d-1)+s\eta}2^{j-s-k}\C_\Om.
\Ee

\medskip

\textbf{Estimate of $A_{j,k,2}^{s,v}$}. Utilizing $|y-y_0|\lc2^{j-s}$ and  $\|\nabla\tilde{\chi}_{J^\iota}\|_{L^\infty(\R^d)}\lc2^{-j+s\varkappa}$ (see \eqref{e:27parix}), we get
$$|\tilde{\chi}_{J^\iota}(y+r\tet)-\tilde{\chi}_{J^\iota}(y_0+r\tet)|\leq |y-y_0|\cdot\|\nabla \tilde{\chi}_{J^\iota}\|_{L^\infty(\R^d)}\lc2^{-(1-\varkappa)s}.
$$

Since $2^{j-4}\leq r\leq2^{j-2}$, we obtain
$|\varphi_j(r)r^{-1}|\lc 2^{-j}.$
It is easy to see that
$$|(I-2^{-2k}\Delta_\xi)^N\psi_k(\xi)|\lc 1.$$
Since we need to get the $L^1$ estimate of $A_{j,k,2}^{s,v}(\cdot,y)$, note that by the support of $\psi_k(\xi)$, we have
$$\int_{|\xi|\lc2^{-k}}\int\Big(1+2^{-2k}|x-y-r\theta|^2\Big)^{-N}dxd\xi\lc 1.$$
Integrating in $r$ yields a bound $2^j$. Recall that $\|\Om\|_\infty\leq 2^{s\eta}\|\Om\|_1$, so integrating in $\theta$, we get a bound $2^{-s\ga(d-1)+s\eta}\C_\Omega$.
Combining these bounds, we can get
\Be\label{e:27Ajk2}
\|A_{j,k,2}^{s,v}(\cdot,y)\|_{L^1(\R^d)}\lc2^{-s\ga(d-1)+s\eta}2^{-(1-\varkappa)s}\C_\Om.
\Ee
\medskip

\textbf{Estimate of $A_{j,k,3}^{s,v}$}.
For the term $A_{j,k,3}^{s,v}(\cdot,y)$, we can deal with it in the similar way as $A_{j,k,1}^{s,v}(\cdot,y)$ once we
have the following observation
\begin{equation*}
\begin{split}
\Big|\Psi(y)-\Psi(y_0)\Big|&=\Big|\int_0^1\big\langle y-y_0,
\nabla \Psi(ty+(1-t)y_0)\big\rangle dt
\Big|\\
&\lc|y-y_0|2^{-k}\int_0^1\fr{N2^{-k}|x-(ty+(1-t)y_0)-r\theta|}{(1+2^{-2k}|x-(ty+(1-t)y_0)-r\theta|^2)^{N+1}}dt
\end{split}
\end{equation*}
where $\Psi(y)=(1+2^{-2k}|x-y-r\theta|^2)^{-N}$. Since $y\in Q$ and $y_0$ is the center of $Q$, we get $|y-y_0|\lc 2^{j-s}$. Because $2^{j-4}\leq r\leq2^{j-2}$, it follows that
$|\varphi_j(r)r^{-1}|\lc 2^{-j}.$  Moreover, we have
$$|(I-2^{-2k}\Delta_\xi)^N\psi_k(\xi)|\lc 1.$$
To get the $L^1$ estimate of $A_{j,k,3}^{s,v}(\cdot,y)$, note that by the support of $\psi_k(\xi)$, we obtain
$$\int_{|\xi|\lc2^{-k}}\int\fr{N2^{-k}|x-(ty+(1-t)y_0)-r\theta|}{(1+2^{-2k}|x-(ty+(1-t)y_0)-r\theta|^2)^{N+1}}dxd\xi\lc 1.$$
Integrating in $r$, we get a bound $2^j$. Integrating in $t$ gives a finite bound $1$. Using the assumption $\|\Om\|_\infty\leq 2^{s\eta}\|\Om\|_1$,  integrating in $\theta$ then gives a bound $2^{-s\ga(d-1)+s\eta}\C_\Om$.
Combining these bounds, we obtain
\Be\label{e:27Ajk3}
\|A_{j,k,3}^{s,v}(\cdot,y)\|_{L^1(\R^d)}\lc2^{-s\ga(d-1)+s\eta}2^{j-s-k}\C_\Om.
\Ee

\medskip

Finally we conclude that the required estimate \eqref{e:27lamkTvj} follows from \eqref{e:27A_m}, \eqref{e:27Ajk1}, \eqref{e:27Ajk2} and \eqref{e:27Ajk3}. Hence we complete the proof of Lemma \ref{l:27l^1_2}.
\end{proof}
\vskip0.24cm

\subsection{Proof of Lemma \ref{l:27finall24}}\quad
\vskip0.24cm

Let us come back to the proof of Lemma \ref{l:27finall24}. We first give an $L^1$ estimate with an exponential decay in $s$.  Let $\eps\in(0,1)$ be a constant to be chosen later. By the triangle inequality, \eqref{e:27sumk} and Lemma \ref{l:27IPjslamk}, we derive

\Be\label{e:27sumkfinal}
\begin{split}
&\quad\Big\|\sum_{j\in\Z}\sum_{Q\in\mathfrak{Q}_{j-s}}\sum_v(I-P_{j-s\varkappa})(I-G^s_v)[\epsilon_J(T^v_j b_{Q})\tilde{\chi}_{J^\iota}]\Big\|_{L^1(\R^d)}\\
&\lc\sum_{j\in\Z}\sum_{Q\in\mathfrak{Q}_{j-s}}\sum_v\Big[\sum_{k\leq j-s\eps}+\sum_{k> j-s\eps}2^{j-k-s\varkappa}\Big]\|(I-G^s_v)\Lam_k[\epsilon_J(T^v_j b_{Q})\tilde{\chi}_{J^\iota}]\|_{L^1(\R^d)}.
\end{split}
\Ee
Next utilizing Lemma \ref{l:27l^1_1} with $N=[\fr{d}{2}]+1$ for $\sum_{k\leq j-s\eps}$ and Lemma \ref{l:27l^1_2} for $\sum_{k>j-s\eps}$, together with $\card(\Theta_s)\lc2^{s\ga(d-1)}$ and \eqref{e:27sumbQ},  we get \eqref{e:27sumkfinal} is majorized by

\begin{equation}\label{e:27sumkfinalss}
\begin{split}
\sum_{j\in\Z}\sum_{Q\in\mathfrak{Q}_{j-s}}&(2^{-\varrho_1s}+2^{-\varrho_2s}+2^{-\varrho_3s})\C_\Omega\|b_Q\|_{L^1(\R^d)}\\
&\lc(2^{-\varrho_1s}+2^{-\varrho_2s}+2^{-\varrho_3s})2^{-2n}\C_\Omega\|f\|_{L^1(\R^d)}
\end{split}
\end{equation}
where
\begin{equation*}
\begin{split}
\varrho_1&=\eps N_1-\Big(2\ga([\fr{d}{2}]+1)+\varkappa N_1+\ga N_1+\eta\Big),\
\varrho_2=1-(\eps+\eta)
\end{split}
\end{equation*}
and
$$\varrho_3=1-(2\eps+\eta-\varkappa).$$

On the other hand, by Lemma \ref{l:27L^33} and \eqref{e:27sumbQ}, we get
\Be\label{e:27sumkfinal2}
\begin{split}
    \Big\|\sum_{j\in\Z}\sum_{Q\in\mathfrak{Q}_{j-s}}\sum_v&(I-P_{j-s\varkappa})(I-G^s_v)[\epsilon_J(T^v_j b_{Q})\tilde{\chi}_{J^\iota}]
\Big\|_{L^3(\mathbb{R}^d)}\\
&\lc2^{s[\ga([\fr{d}{2}]+1)+\fr23\gamma(d-1)+\eta]}\Big(\lambda^2\C_\Omega2^{-2n}\|f\|_{L^1(\R^d)}\Big)^{1/3}.
\end{split}
\Ee

Making an interpolation between \eqref{e:27sumkfinal} with \eqref{e:27sumkfinalss} and \eqref{e:27sumkfinal2}, we derive
\Bes
\begin{split}
    \Big\|\sum_{j\in\Z}\sum_{Q\in\mathfrak{Q}_{j-s}}\sum_v&(I-P_{j-s\varkappa})(I-G^s_v)[\epsilon_J(T^v_j b_{Q})\tilde{\chi}_{J^\iota}]
\Big\|_{L^2(\mathbb{R}^d)}\\
&\lc\Big(2^{-s\vartheta_1}+2^{-s\vartheta_2}+2^{-s\vartheta_3}\Big)\Big(\lam \C_\Om2^{-2n}{\|f\|_{L^1(\R^d)}}\Big)^{1/2}
\end{split}\Ees
where
$$\vartheta_1=\fr14\Big[\eps N_1-\Big(2\ga([\fr{d}{2}]+1)+\varkappa N_1+\ga N_1+\eta\Big)\Big]-\fr34\Big[\ga([\fr{d}{2}]+1)+\fr23\gamma(d-1)+\eta\Big],$$
$$\vartheta_2=\fr14\Big[1-(\eps+\eta)\Big]-\fr34\Big[\ga([\fr{d}{2}]+1)+\fr23\gamma(d-1)+\eta\Big],$$
and
$$\vartheta_3=\fr14\Big[1-(2\eps+\eta-\varkappa)\Big]-\fr34\Big[\ga([\fr{d}{2}]+1)+\fr23\gamma(d-1)+\eta\Big].
$$

We now select parameters satisfying $0<\eta\ll\ga\ll\varkappa\ll\eps\ll1$ and choose an integer $N_1>0$ sufficiently large to ensure $\vartheta_1>0$, $\vartheta_2>0$ and $\vartheta_3>0$.
It should be pointed out that these parameters $\eta, \gamma, \varkappa, \varepsilon$ are chosen consistently with our earlier arguments.
Consequently by choosing the constant $\del_4$ such that
$$\del_4=\min\big\{2\vartheta_1,2\vartheta_2,2\vartheta_3\big\},$$
we complete the proof of Lemma \ref{l:27finall24}.
$\hfill{} \Box$

\vskip0.24cm

%\section*{Declarations}

%\subsection*{Ethical Approval}
%The declaration for ethical approval is not applicable.

%\subsection*{Competing interests} The author declares no competing interests.

%\subsection*{Data Availability} No datasets were generated or analysed during the current study.

\vskip0.24cm
\bibliographystyle{amsplain}
\bibliography{rf}

\end{document}